\documentclass[dvipsnames]{siamart190516}

\title{Sensitivity of matrix function based network communicability measures: Computational methods and a priori bounds}

\author{M. Schweitzer\thanks{School of Mathematics and Natural Sciences, Bergische Universit\"at Wuppertal, 42097 Wuppertal, Germany, \email{marcel@uni-wuppertal.de}.}}

\usepackage{lineno,hyperref}
\modulolinenumbers[5]
\usepackage{amsmath}
\usepackage{amsfonts,amssymb}
\usepackage{graphicx}
\usepackage{bm}
\graphicspath{{./Figures/}}
\usepackage{color}
\usepackage{ulem}
\usepackage{multirow}
\usepackage{multicol}
\usepackage{algorithm}
\usepackage{algpseudocode}
\usepackage{algorithmicx}
\usepackage{hyperref}
\usepackage{upquote}
\usepackage{cprotect}
\usepackage{setspace}

\usepackage[showonlyrefs]{mathtools}
\usepackage{xcolor}

\usepackage[defaultcolor=red]{changes}

\DeclareMathAlphabet{\mathbf}{OT1}{cmr}{bx}{it}
\newcommand{\va}{{\mathbf a}}
\newcommand{\vb}{{\mathbf b}}
\newcommand{\vc}{{\mathbf c}}

\newcommand{\ve}{{\mathbf e}}

\newcommand{\vg}{{\mathbf g}}

\newcommand{\vu}{{\mathbf u}}
\newcommand{\vv}{{\mathbf v}}

\newcommand{\vw}{{\mathbf w}}
\newcommand{\vx}{{\mathbf x}}
\newcommand{\vy}{{\mathbf y}}

\newcommand{\vnull}{\boldsymbol{0}}
\newcommand{\vone}{\boldsymbol{1}}
\newcommand{\spK}{{\cal K}}

\newcommand{\lmin}{\lambda_{\min}}
\newcommand{\lmax}{\lambda_{\max}}

\renewcommand{\d}{\,\mathrm{d}}
\newcommand{\N}{\mathbb{N}}
\newcommand{\R}{\mathbb{R}}
\newcommand{\Rn}{\mathbb{R}^n}
\newcommand{\Rnn}{\mathbb{R}^{n \times n}}
\newcommand{\Rmm}{\mathbb{R}^{m \times m}}

\newcommand{\Cnn}{\mathbb{C}^{n \times n}}

\newcommand{\V}{\mathcal{V}}

\newcommand{\E}{\mathcal{E}}

\DeclareMathOperator{\spec}{spec}
\DeclareMathOperator{\trace}{tr}

\DeclareMathOperator{\TN}{TN}

\DeclareMathOperator{\tn}{TN}
\DeclareMathOperator{\SC}{SC}
\DeclareMathOperator{\EE}{EE}

\DeclareMathOperator{\vecop}{vec}

\newtheorem{remarksimple}[theorem]{Remark}
\let\oldremarksimple\remarksimple
\renewcommand{\remarksimple}{\oldremarksimple\normalfont}
\newenvironment{remark}{\begin{remarksimple}}{\hfill$\diamond$\end{remarksimple}}
\newtheorem{experimentsimple}[theorem]{Experiment}
\let\oldexperimentsimple\experimentsimple
\renewcommand{\experimentsimple}{\oldexperimentsimple\normalfont}

\newtheorem{examplesimple}[theorem]{Example}
\let\oldexamplesimple\examplesimple
\renewcommand{\examplesimple}{\oldexamplesimple\normalfont}
\newenvironment{example}{\begin{examplesimple}}{\hfill$\diamond$\end{examplesimple}}

\usepackage{pgfplots}
\usepgfplotslibrary{groupplots}
\usetikzlibrary{external}

\pgfkeys{/pgf/images/include external/.code={\includegraphics{#1}}}

\usetikzlibrary{decorations.markings}
\makeatletter
\tikzset{
  nomorepostactions/.code={\let\tikz@postactions=\pgfutil@empty},
  mymark/.style 2 args={decoration={markings,
    mark= between positions 0 and 1 step (1/9)*\pgfdecoratedpathlength with{%
        \tikzset{#2,every mark}\tikz@options
        \pgfuseplotmark{#1}%
      },  
    },
    postaction={decorate},
    /pgfplots/legend image post style={
        mark=#1,mark options={#2},every path/.append style={nomorepostactions}
    },
  },
}
\makeatother

\pgfplotscreateplotcyclelist{list_std}{%
Cerulean,line width=1.25pt, mark=diamond, solid\\
BurntOrange,line width=1.25pt, mark=*, dashed\\
OliveGreen,line width=1.25pt, mark=x, dash-dotted\\
Orchid,line width=1.25pt, mark=square, dotted\\
}
\pgfplotsset{compat=1.17}

\begin{document}
\maketitle

\pagestyle{myheadings} \thispagestyle{plain}
\markboth{M. SCHWEITZER}{SENSITIVITY OF COMMUNICABILITY MEASURES}

\begin{abstract}
When analyzing complex networks, an important task is the identification of those nodes which play a leading role for the overall communicability of the network. In the context of modifying networks (or making them robust against targeted attacks or outages), it is also relevant to know how sensitive the network's communicability reacts to changes in certain nodes or edges. Recently, the concept of \emph{total network sensitivity} was introduced in [O.\ De la Cruz Cabrera, J.\ Jin, S.\ Noschese, L.\ Reichel, \emph{Communication in complex networks}, Appl.\ Numer.\ Math., 172, pp.\ 186--205, 2022], which allows to measure how sensitive the total communicability of a network is to the addition or removal of certain edges. One shortcoming of this concept is that sensitivities are extremely costly to compute when using a straight-forward approach (orders of magnitude more expensive than the corresponding communicability measures). In this work, we present computational procedures for estimating network sensitivity with a cost that is essentially linear in the number of nodes for many real-world complex networks. Additionally, we extend the sensitivity concept such that it also covers sensitivity of subgraph centrality and the Estrada index, and we discuss the case of node removal. We propose a priori bounds for these sensitivities which capture well the qualitative behavior and give insight into the general behavior of matrix function based network indices under perturbations. These bounds are based on decay results for Fr\'echet derivatives of matrix functions with structured, low-rank direction terms which might be of independent interest also for other applications than network analysis.
\end{abstract}



\section{Introduction}\label{sec:intro}
Complex networks---mathematically modeled by \emph{graphs} consisting of \emph{nodes} and \emph{edges}---occur as models in a wide range of application areas, including, but not limited to, biology, chemistry, life sciences, social sciences and humanities~\cite{Barabasi2003,BoccalettiLatoraMorenoChavezHwang2006,Estrada2012,Newman2018}. Central tasks in analyzing complex networks are identifying the most important (or \emph{central}) nodes in the network and measuring the overall \emph{communicability} of the network. Additionally, in particular when designing networks, it can be of interest to investigate how centrality and communicability react to \emph{modifications of the network}. This helps, e.g., to answer questions of robustness or vulnerability of a network with respect to outages or targeted attacks~\cite{CohenHavlin2010}, or it can help to decide how to best enhance/augment (by introducing additional connections) or optimize/trim (by removing unneeded edges) the network~\cite{ArrigoBenzi2016a,ArrigoBenzi2016b}. 

Many important network centrality indices and communicability measures are based on matrix functions, in particular the matrix exponential~\cite{BenziEstradaKlymko2013,BenziKlymko2013,DelaCruzCabreraMatarReichel2019,DeLaCruzCabreraMatarReichel2021,EstradaRodriguezVelaszquez2005,EstradaHatano2008,EstradaHigham2010} and the resolvent~\cite{Katz1953}, but sometimes also more general functions~\cite{ArrigoDurastante2021,BenziKlymko2015}; see also the recent survey~\cite{BenziBoito2020}. In the last few years, many different approaches have been developed which are intimately related to the question how these matrix function based centrality indices react to changes in the network. These include, e.g., algorithms for (near-)optimally up/downdating networks based on well-chosen heuristics~\cite{ArrigoBenzi2016a,ArrigoBenzi2016b}, algorithms for efficiently updating matrix functions under general low-rank changes~\cite{BeckermannKressnerSchweitzer2018,BeckermannCortinovisKressnerSchweitzer2021}, stability estimates based on decay bounds for matrix functions~\cite{PozzaTudisco2018} as well as sensitivity measures defined in terms of the Fr\'echet derivative~\cite{DelaCruzCabreraJinNoscheseReichel2021}. 

In this work, we develop computational procedures for approximating the Fr\'echet derivative based sensitivities that are asymptotically much more efficient than the methods originally proposed in~\cite{DelaCruzCabreraJinNoscheseReichel2021}, thus making it feasible to use the sensitivity for ranking the importance of edges also in large-scale networks. In particular, by leveraging an approach from~\cite{HighamRelton2016} for estimating the largest elements of an implicitly given matrix in combination with a Krylov subspace method proposed in~\cite{KandolfKoskelaReltonSchweitzer2021, Kressner2019}, we demonstrate how one can efficiently identify the edges with respect to which the network is most sensitive also in situations where it is not possible to explicitly compute or store all edge sensitivities.

We then extend the sensitivity concept by generalizing the measures from~\cite{DelaCruzCabreraJinNoscheseReichel2021} in several ways. On the one hand, we generalize from \emph{total communicability} to other frequently used measures like \emph{subgraph centrality} and the \emph{Estrada index}, and on the other hand, we extend it to also cover \emph{node} modifications in addition to \emph{edge} modifications. 

Additionally, we derive bounds for the decay in Fr\'echet derivatives with structured, low-rank direction terms and use these to obtain a priori bounds for network sensitivity measures, similar in spirit to the results of~\cite{PozzaTudisco2018}. These bounds mathematically confirm the intuition that nodes which are nearby a modified edge or node are more sensitive to the modification than nodes which are farther away. 

The remainder of the paper is organized as follows. In Section~\ref{sec:basics} we recall some basic facts about graphs, matrix functions and communicability measures and we brief\/ly review the concept of total network sensitivity introduced in~\cite{DelaCruzCabreraJinNoscheseReichel2021}. Additionally, we prove our first main result (Theorem~\ref{the:frechet_eij} and Corollary~\ref{cor:total_sensitivity_simple_formula}), which forms the basis for a large part of the developments in this manuscript. In Section~\ref{sec:algorithms}, we propose an algorithm for approximating network sensitivities, analyze its computational cost and demonstrate its viability on real-world networks. Section~\ref{sec:sensitivity} deals with the extension of the sensitivity concept to subgraph centrality and the Estrada index and to node modifications. In Section~\ref{sec:decay_bounds} we first derive results about the nonzero pattern of Fr\'echet derivatives of polynomial matrix functions with structured, low-rank direction terms and then use these to obtain a priori bounds on network sensitivity. Concluding remarks are given in Section~\ref{sec:conclusions}. Some technical proofs are collected in Appendix~\ref{sec:appendix_proofs}.

\section{Basics \& Notation}\label{sec:basics}

In this section, we recall some basic definitions and fix our notation.

\subsection{Notation}\label{subsec:notation}
We denote by $\ve_i \in \Rn$ the $i$th canonical unit vector and by $\vone = [1, \dots, 1]^T \in \Rn$ the vector of all ones. The $(i,j)$th entry of a matrix function $f(A)$ is denoted by $[f(A)]_{ij}$ and the \emph{trace} of $f(A)$, i.e., the sum of its diagonal entries, is denoted by $\trace(f(A))$. By $\|\cdot\|$ we denote the Euclidean vector norm and the spectral matrix norm it induces. The \emph{spectrum} of a matrix $A$, i.e., the set of all its eigenvalues, is denoted by $\spec(A)$.

\subsection{Graphs and matrices}\label{subsec:graphs_matrices}
A graph $G = (\V, \E)$ is defined by a set $\V$ of \emph{nodes} and a set $\E \subseteq \V \times \V$ of \emph{edges}. For simplicity, we will assume in the following that $\V = \{1,\dots,n\}$. A graph is called \emph{weighted} if it is equipped with a weight function $w: \E \longrightarrow \R^+$ that assigns a weight $w_{ij} > 0$ to each edge $(i,j) \in \E$. An unweighted graph can be interpreted as a weighted graph with all edge weights equal to $1$. If the set $\E$ is such that $(i,j) \in \E$ if and only if $(j,i)\in \E$, and $w_{ij} = w_{ji}$, then $G$ is called \emph{undirected graph}, otherwise it is called \emph{directed graph} (or \emph{digraph}). In the following, we always tacitly assume that $G$ contains no self loops (i.e., edges connecting a node to itself) and no multiple edges with the same direction between any two nodes. By $d(u,v)$, we denote the \emph{geodesic distance} in $G$, i.e., the smallest number of edges on any path connecting node $u$ to node $v$.

A (weighted) graph $G$ can be represented by the \emph{(weighted) adjacency matrix} $A_G \in \Rnn$ defined via
\begin{equation*}
    a_{ij} = 
    \begin{cases}
    w_{ij} & \text{ if } (i,j) \in \E\\
    0 & \text{ otherwise.}
    \end{cases}
\end{equation*}
Clearly, when $G$ is undirected, the adjacency matrix $A_G$ is symmetric.

Note that for naming nodes in a graph, we adopt the following convention: When an \emph{edge} is at the center of attention, then we denote its end nodes by $i$ and $j$. When \emph{nodes} themselves are at the center of attention, we denote them by $u$ or $v$.

\subsection{Functions of matrices and the Fr\'echet derivative}\label{subsec:matfun}
Matrix functions can be defined in many different ways. The three most popular ones are based on the Jordan canonical form, Hermite interpolating polynomials and the Cauchy integral formula; see~\cite[Section~1.2]{Higham2008} for a thorough treatment. 

As we are mostly interested in the exponential function in this work, we only recall the third of the above alternatives, which applies to functions $f$ which are analytic on a region that contains $\spec(A)$. In this case $f(A)$ can be defined via the Cauchy integral formula, 
\[
    f(A) := \frac{1}{2\pi i}\int_{\Gamma} f(\zeta)(\zeta I - A)^{-1} \d\zeta,
\]
where $\Gamma$ is a path that winds around $\spec(A)$ exactly once. For a matrix function $f$, the Fr\'echet derivative at the matrix $A$ is an operator $L_f(A, \cdot)$ which is linear (in the second argument) and satisfies
\begin{equation*}
  f(A + E) - f(A) = L_f(A,E) + o(\|{E}\|), \quad \textnormal{for all } E \in \Rnn.
\end{equation*} 

A sufficient condition for $L_f(A,\cdot)$ to exist is that $f$ is $2n-1$ times continuously differentiable on a region containing $\spec(A)$ (see~\cite[Theorem~3.8]{Higham2008}), and if the Fr\'echet derivative exists, it is unique. In particular, the Fr\'echet derivative of a matrix function is guaranteed to exist if $f$ is analytic on a region containing $\spec(A)$, and in this case it has the integral representation
\begin{equation}\label{eq:frechet_derivative_integral}
L_f(A,E) = \frac{1}{2\pi i} \int_\Gamma f(\zeta)(\zeta I - A)^{-1}E(\zeta I - A)^{-1}\d \zeta,
\end{equation}
where $\Gamma$ is again a path that winds around $\spec(A)$; see, e.g.,~\cite{Higham2008,KandolfRelton2017}. Clearly, the Fr\'echet derivative of the exponential function, $L_{\exp}(A,\cdot)$, which is of particular importance for this work, is guaranteed to exist for \emph{any} matrix $A$.

Related is the \emph{G\^ateaux (or directional) derivative} of $f$ at $A$, defined as
\begin{equation*}
    G_f(A,E) = \lim\limits_{h\rightarrow 0} \frac{f(A+hE)-f(A)}{h}.
\end{equation*}
If $L_f(A,\cdot)$ exists, then it is equal to $G_f(A,\cdot)$, but the converse is not necessarily true; even when all directional derivatives of $f$ at $A$ exist, $f$ does not need to be Fr\'echet differentiable at $A$. However, for the exponential function, both derivatives are guaranteed to exist and coincide, i.e.,
\begin{equation}\label{eq:frechet_gateaux}
    L_{\exp}(A,E) = G_{\exp}(A,E), \quad \textnormal{ for all } A, E \in \Rnn.
\end{equation}

The following important identity for the Fr\'echet derivative reduces its computation to that of a (block triangular) matrix function of twice the size:
\begin{equation*}
\arraycolsep=8pt
f\left(\left[\begin{array}{cc} A & E \\ 0 & A \end{array}\right]\right) = \left[\begin{array}{cc}f(A) & L_f(A,E) \\ 0 & f(A)\end{array}\right].
\end{equation*}

Associated with the Fr\'echet derivative is its \emph{Kronecker form} $K_f(A) \in \R^{n^2 \times n^2}$, the matrix which fulfills
\begin{equation}\label{eq:kronecker_form}
    K_f(A)\vecop(E) = \vecop(L_f(A,E)) \quad\text{ for all } E \in \Rnn,
\end{equation}
where $\vecop(M)$ stacks the columns of $M \in \Rnn$ into a vector of length $n^2$. Due to its very large size, the Kronecker form is seldomly used in actual computations, but it can be a useful theoretical tool.

\subsection{Communicability measures and total network sensitivity}\label{subsec:communicability_measures}
In many applications, it is important to measure how well information can spread through a network. A frequently used measure for this is \emph{total communicability}~\cite{BenziKlymko2013}, defined as
\begin{equation}\label{eq:total_communicability}
    C^{\tn}(A_G) = \vone^T\exp(A_G)\vone.
\end{equation}
For judging the importance/centrality of an individual node $v$ in the network, one often uses (exponential) subgraph centrality~\cite{EstradaRodriguezVelaszquez2005}, defined as
\begin{equation}\label{eq:exponential_subgraph_centrality}
    \vc^{\SC}(v) := \ve_v^T\exp(A_G)\ve_v.
\end{equation}
A related measure for the overall communicability of the network---and thus an alternative to total communicability~\eqref{eq:total_communicability}---is the \emph{Estrada index}, the sum of all subgraph centralities, i.e.,
\begin{equation}\label{eq:estrada_index}
    \EE(A_G) = \sum_{v = 1}^n \vc^{\SC}(v) = \trace(\exp(A_G));
\end{equation}
see, e.g.,~\cite{Estrada2012, EstradaHatano2008}. 

It is also often of interest how strongly the communicability of a network is \emph{affected by modifications} of the network~\cite{ArrigoBenzi2016a, ArrigoBenzi2016b, DelaCruzCabreraJinNoscheseReichel2021, PozzaTudisco2018}. The type of network modification that is most frequently considered in this context is the addition or removal of edges. In~\cite{DelaCruzCabreraJinNoscheseReichel2021} the concept of total network sensitivity is introduced, which measures how sensitive total communicability~\eqref{eq:total_communicability} is with respect to modification of one specific edge.

\begin{definition}\label{def:sensitivity}
Let $G = (\V, \E, w)$ be a weighted (di)graph with adjacency matrix $A_G \in \Rnn$ and let $E_{ij} := \ve_i\ve_j^T \in \Rnn$. Then the \emph{total network sensitivity of $G$ with respect to changes in $w_{ij}$} is defined in terms of the Fr\'echet derivative $L_{\exp}(A_G, E_{ij})$ as
\begin{equation}\label{eq:network_sensitivity}
S^{\tn}_{ij}(A_G) := \vone^TL_{\exp}(A_G, E_{ij})\vone.
\end{equation}
\end{definition}

\begin{remark}\label{rem:undirected}
When $G$ is undirected, it seems natural to define the sensitivity with respect to changes in $w_{ij}$ using the Fr\'echet derivative with respect to the symmetric rank-two direction term $E = \ve_i\ve_j^T + \ve_j\ve_i^T$. However, as the Fr\'echet derivative is linear in its second argument, we have $L_{\exp}(A_G, \ve_i\ve_j^T + \ve_j\ve_i^T) = L_{\exp}(A_G, \ve_i\ve_j^T) + L_{\exp}(A_G, \ve_j\ve_i^T)$. Further, by elementary properties, $L_{\exp}(A_G, E) = L_{\exp}(A_G, E^T)^T$ when $A_G$ is symmetric. We thus have
\begin{equation}\label{eq:transposition_invariance}
\vone^T \cdot L_{\exp}(A_G, \ve_i\ve_j^T + \ve_j\ve_i^T) \cdot \vone = 2 \cdot \vone^T \cdot L_{\exp}(A_G, \ve_i\ve_j^T) \cdot \vone
\end{equation}
Thus, in light of~\eqref{eq:transposition_invariance}
it also suffices to consider rank-one direction terms in the undirected case.
\end{remark}

The following relationship opens the door for efficiently computing total network sensitivity for large scale networks; cf.~Section~\ref{sec:algorithms}. We first formulate it for general analytic $f$, as it might be of independent interest also in other application areas and then state the formulation as needed in our setting.

\begin{theorem}\label{the:frechet_eij}
Let $A \in \Rnn, \vu, \vv\in \Rn$, let $f$ be Fr\'echet differentiable at $A$ and denote $E_{ij} = \ve_i\ve_j^T$. Then
\begin{equation}\label{eq:the_frechet_eij}
\vu^T L_f(A,E_{ij})\vv = [L_{f}(A^T,\vu\vv^T)]_{ij}.
\end{equation}
\end{theorem}
\begin{proof}[Proof of Theorem~\ref{the:frechet_eij}]
We start by vectorizing the left-hand side of~\eqref{eq:network_sensitivity}, noting that $\vecop(\alpha) = \alpha$ for any scalar $\alpha \in \R$, which yields
\begin{equation}\label{eq:proof_thm1}
\vu^TL_{f}(A, E_{ij})\vv = \vecop(\vu^TL_{f}(A, E_{ij})\vv) = (\vv^T \otimes \vu^T)\vecop(L_{f}(A,E_{ij})),
\end{equation}
where we have used the well-known relation $\vecop(BCD) = (D^T \otimes B)\vecop(C)$ for the second equality. Now, by inserting the definition~\eqref{eq:kronecker_form} of the Kronecker form of the Fr\'echet derivative into~\eqref{eq:proof_thm1}, we further have
\begin{equation}\label{eq:proof_simple_formula_total1}
\vu^TL_{f}(A, E_{ij})\vv = (\vv^T \otimes \vu^T)K_{f}(A)\vecop(E_{ij}) = (\vv^T \otimes \vu^T)K_{f}(A)\ve_{(j-1)n+i},
\end{equation}
as the vectorization of a matrix $E_{ij}$ with just a single entry 1 results in a canonical unit vector. Further, $\vecop(\vu\vv^T) = (\vv \otimes \vu)$, and $\alpha^T = \alpha$ for any scalar $\alpha$, so that by taking the transpose of~\eqref{eq:proof_simple_formula_total1} and noting that $K_{f}(A)^T = K_{f}(A^T)$, we obtain
\begin{align*}
   \vu^T L_f(A,E_{ij})\vv   &= \ve_{(j-1)n+i}^TK_{f}(A^T)(\vv \otimes \vu) \\
                            &= \ve_{(j-1)n+i}^T\vecop(L_{f}(A^T,\vu\vv^T)) \\
                            &= [L_{f}(A^T,\vu\vv^T)]_{ij},
\end{align*}
which concludes the proof.
\end{proof}

\begin{corollary}\label{cor:total_sensitivity_simple_formula}
Let $S_{ij}^{\TN}(A_G)$ denote total network sensitivity, defined in~\eqref{eq:network_sensitivity}. Then
\begin{equation}\label{eq:total_sensitivity_simple_formula}
S_{ij}^{\TN}(A_G) = [L_{\exp}(A_G^T,\vone\vone^T)]_{ij}.
\end{equation}
\end{corollary}
\begin{proof}
The result directly follows by applying Theorem~\ref{the:frechet_eij} to $f(A) = \exp(A_G)$ and $\vu = \vv = \vone$.
\end{proof}

The advantage of~\eqref{eq:total_sensitivity_simple_formula} over~\eqref{eq:network_sensitivity} is that it characterizes the network sensitivities $S_{ij}^{\TN}(A_G)$ with respect to all possible edge modifications as entries of \emph{a single Fr\'echet derivative}, while in the original formulation, \emph{the direction term changes} depending on the edge under consideration.

\section{Efficiently computing sensitivity measures}\label{sec:algorithms}
In this section, we discuss an algorithm for approximating total network sensitivity and in particular how to efficiently find the edge modifications with respect to which the network is most sensitive.

We begin by recapitulating a Krylov subspace method for approximating Fr\'echet derivatives with low-rank direction terms from~\cite{KandolfKoskelaReltonSchweitzer2021,Kressner2019} in Section~\ref{subsec:krylov}. This method forms a basic building block of our final algorithm (as well as of the original algorithm for network sensitivity from~\cite{DelaCruzCabreraJinNoscheseReichel2021}). 

\subsection{A basic Krylov subspace scheme for Fr\'echet derivatives}\label{subsec:krylov}
As it does not complicate the exposition, we consider the case of approximating the Fr\'echet derivative with respect to a general rank-one direction term $E = \vb\vc^T$ in the following, although we are mostly interested in direction terms with very specific structure. Without loss of generality, we further assume that $\|\vb\|=\|\vc\| = 1$.

By Corollary~\ref{cor:total_sensitivity_simple_formula}, total network sensitivity~\eqref{eq:total_sensitivity_simple_formula} with respect to all possible edge modifications can be obtained by computing a Fr\'echet derivative at $A_G^T$ with respect to a rank-one direction term.

To approximate $L_{\exp}(A_G^T, \vb\vc^T)$, the method introduced in~\cite{KandolfKoskelaReltonSchweitzer2021, Kressner2019} first computes orthonormal bases $V_m$, $W_m$ of the two Krylov subspaces $\spK_m(A_G^T, \vb)$ and $\spK_m(A_G, \vc)$ by the Arnoldi method~\cite{Arnoldi1951}, yielding Arnoldi decompositions
\begin{eqnarray}
A_G^TV_m &=& V_mG_m + g_{m+1,m} \vv_{m+1}\ve_m^T \label{eq:arnoldi_relation1},\\
A_GW_m &=& W_mH_m + h_{m+1,m} \vw_{m+1}\ve_m^T \label{eq:arnoldi_relation2},
\end{eqnarray}
where we assume that no breakdown occurs. Note that it is also possible to build two Krylov spaces of different dimensions $m_1 \neq m_2$, respectively. An approximation for $L_{\exp}(A_G^T,\vb\vc^T)$ is then extracted from the tensorized Krylov subspace $\spK_m(A_G, \vc) \otimes \spK_m(A_G^T,\vb)$ as
\begin{equation}\label{eq:Lm}
    L_m = V_mX_mW_m^T,
\end{equation}
where $X_m$ is obtained as the upper right block of a $2m \times 2m$ matrix function,
\begin{equation}\label{eq:krylov_frechet_block}
\arraycolsep=8pt
\exp\left(\left[\begin{array}{cc}G_m & \ve_1\ve_1^T\\ 0 & H_m^T\end{array}\right]\right) = \left[\begin{array}{cc}\exp(G_m) & X_m \\ 0 & \exp(H_m^T)\end{array}\right].
\end{equation}

\begin{algorithm}
\caption{Krylov method for computing the Fr\'echet derivative $L_{\exp}(A^T, \vb\vc^T)$ \label{alg:krylov}}
\begin{algorithmic}[1]

\smallskip

\Statex \textbf{Input:} $A_G \in \R^{n \times n}, \vb, \vc \in \R^n, m \in \N$
\Statex \textbf{Output:} Low-rank factors $V_m, W_m \in \R^{n \times m}, X_m \in \Rmm$ according to~\eqref{eq:Lm}

\smallskip

\State Compute $V_m, G_m$ by $m$ steps of Arnoldi for $A_G^T$ and $\vb$
\State Compute $W_m, H_m$ by $m$ steps of Arnoldi for $A_G$ and $\vc$
\State $C_m \leftarrow \exp\left(\begin{bmatrix} G_m & \ve_1\ve_1^T\\ 0 & H_m^T\end{bmatrix}\right)$
\State $X_m \leftarrow C_m(1:n,n+1:2n)$
\end{algorithmic}
\end{algorithm}

We summarize this procedure in Algorithm~\ref{alg:krylov}. Let us brief\/ly comment on its computational cost: For each of the two Arnoldi decompositions~\eqref{eq:arnoldi_relation1}--\eqref{eq:arnoldi_relation2}, $m$ matrix-vector products need to be computed. Assuming that $G$ is a sparse graph with $\mathcal{O}(n)$ edges, this requires $\mathcal{O}(nm)$ arithmetic operations. Additionally, a modified Gram--Schmidt orthogonalization for the $m+1$ basis vectors is necessary, requiring $\mathcal{O}(nm^2)$ operations. If $G$ is undirected, so that $A_G$ is symmetric, this cost reduces to $\mathcal{O}(nm)$ if no reorthogonalization is performed. Evaluating the matrix function~\eqref{eq:krylov_frechet_block} has a cost of $\mathcal{O}(m^3)$. For $m \ll n$, the overall computational cost for the Krylov method outlined above is therefore given by $\mathcal{O}(nm^2)$ if $G$ is directed and $\mathcal{O}(nm)$ if $G$ is undirected. Note that we have so far omitted the cost for explicitly forming $L_m$ via~\eqref{eq:Lm}, as this is typically prohibitively expensive: In general, $L_m$ is a dense matrix of size $n \times n$, so explicitly forming it requires $\mathcal{O}(n^2)$ storage and has a computational cost of $\mathcal{O}(n^2m+nm^2)$, both of which are not feasible for large scale networks.

If only a few individual entries of $L_m$ are required, these can be cheaply computed at a cost of $\mathcal{O}(m^2)$ per entry, as summarized in the following proposition.

\begin{proposition}\label{prop:individual_entry}
Let $V_m, W_m, X_m$ be computed as explained above and let $L_m$ be defined via~\eqref{eq:Lm}. Then, the entries of $L_m$ are given by
\begin{equation*}
    [L_m]_{uv} = \sum\limits_{i=1}^m\sum\limits_{j=1}^m [V_m]_{ui}[X_m]_{ij}[W_m]_{vj}.
\end{equation*}
Consequently, given $V_m, W_m, X_m$, computing an individual entry of $L_m$ has computational complexity $\mathcal{O}(m^2)$. 
\end{proposition}
\begin{proof}
The result follows directly from the formula~\eqref{eq:Lm} for the approximation $L_m$ and the rules of matrix-matrix multiplication.
\end{proof}

\subsection{Finding the top few sensitivities}\label{subsec:top_sensitivities}
As already commented at the end of the preceding section, it is not possible to compute or store sensitivities with respect to all possible edge modifications when $G$ is large. While it is indeed cheaply possible to recover individual entries of $L_m$ according to Proposition~\ref{prop:individual_entry}, there is a fundamental flaw in this approach: typical use cases for computing sensitivities are finding a (close to) optimal update of a network or identifying the most vulnerable parts of the network, both of which require identifying a few edges with very high sensitivity values. Thus, while it \emph{is} sufficient to know \emph{just a few} sensitivity values, it is not known a priori \emph{which ones}.

If storage of the dense $n \times n$ matrix $L_m$ is the main concern, then one can use Proposition~\ref{prop:individual_entry} to compute all individual sensitivities one after the other, keeping track of the $p$ largest or smallest values (and their locations), discarding all other sensitivities. This way, the top $p$ sensitivities can be found consuming only a fixed amount of storage, but at the very high computational cost of $\mathcal{O}(m^2n^2)$.

We therefore now highlight a better approach for tackling this problem, leveraging a method from~\cite{HighamRelton2016} for computing the largest elements of an implicitly given matrix $S$, accessing it only via matrix-vector products.  We brief\/ly outline this method in its most basic form, closely following the presentation in~\cite{HighamRelton2016}. For further details and more sophisticated variants, we refer the reader to~\cite[Sections~2,~4 and~5]{HighamRelton2016}.

Assume we want to find the single largest element (in modulus) of the matrix $S$. A first observation is that this can be interpreted as a mixed subordinate norm,
\begin{equation}\label{eq:mixed_subordinate_norm}
    \max_{i,j = 1,\dots,n} |s_{ij}| = \max_{\vx \neq \vnull} \frac{\|S\vx\|_{\infty}}{\|\vx\|_1}.
\end{equation}
Finding the maximum on the right-hand side of~\eqref{eq:mixed_subordinate_norm} can be phrased as the optimization problem
\begin{align}\label{eq:optimization_problem}
\begin{split}
    \max \ \ \ & F(\vx) := \|S\vx\|_\infty\\
    \textnormal{s.t.} \ \ \ & x \in \mathbb{E} := \{\vx : \|\vx\|_1 \leq 1\}.
\end{split}
\end{align}
This is a convex optimization problem, and for any $\vx \in \mathbb{E}$ there exists at least one subgradient, i.e., a vector $\vg$ for which $F(\vy) \geq F(\vx) + \vg^T(\vy-\vx)$ for all $\vy \in \mathbb{E}$. The set of all subgradients of $F$ at $\vx$ is denoted by $\partial F(\vx)$. Further, we denote the \emph{dual set} of a vector $\vx$ by
\begin{equation*}
    \text{dual}_\infty(\vx) := \{ \vy : \vy^T\vx = \|\vx\|_\infty, \ \|\vy\|_1 = 1 \}.
\end{equation*}
Then clearly, a vector $\vy^\ast$ that maximizes $\vg^T(\vy-\vx)$ must fulfill $\vy^\ast \in \text{dual}_\infty(\vg)$ and the set of subgradients at $\vx$ fulfills $\partial F(\vx) \supseteq S^T\text{dual}_\infty(S\vx)$. By some algebraic manipulations, one can show that for maximizing $F$ one can always select a subgradient $\vg \in S^T\text{dual}_\infty(S\vx)$. These observations directly give rise to a method for solving~\eqref{eq:optimization_problem}, which alternatingly selects a subgradient $\vg \in S^T\text{dual}_\infty(S\vx)$ and a point $\vx \in \text{dual}_\infty(\vg)$. This requires performing two matrix-vector products per iteration, one with $S$ and one with $S^T$. We give an algorithmic description of this method as Algorithm~\ref{alg:power}. 

\begin{algorithm}
\caption{Power method for finding largest modulus element of a matrix\label{alg:power}}
\begin{algorithmic}[1]
\setstretch{1.2}

\smallskip

\Statex \textbf{Input:} $S \in \R^{n_1 \times n_2}$
\Statex \textbf{Output:} $\gamma \in \R, \vx \in \R^{n_2}$ s.t. $\gamma \leq \max_{i,j} s_{ij}$ and $\|S\vx\|_\infty = \gamma\|\vx\|_1$

\smallskip

\State $\vx \leftarrow (1/n_2)\vone$
\For{$k = 1, 2, \dots$}
    \State $\vy \leftarrow S\vx$ 
    \If{$k > 1$}
        \If{$\|\vy\|_\infty \leq \|\vg\|_\infty$}
            \State $\gamma = \|\vg\|_\infty$
            \State quit
        \EndIf
    \EndIf
    \State Select smallest $i$ such that $|\vy_i| = \|\vy\|_\infty$
    \State $\vg \leftarrow S^T\ve_i$
    \If{$\|\vy\|_\infty \leq \|\vg\|_\infty$}
        \State $\gamma = \|\vg\|_\infty$
        \State quit
    \EndIf
    \State Select smallest $j$ such that $|\vg_j| = \|\vg\|_\infty$
    \State $\vx \leftarrow \ve_j$
\EndFor

\end{algorithmic}
\end{algorithm}

Let us note that Algorithm~\ref{alg:power} was (in a similar form) already proposed in~\cite{Boyd1974,Tao1984} before~\cite{HighamRelton2016}. However, in~\cite{HighamRelton2016}, the concept is extended in several ways, by introducing a blocked version of the algorithm and (by using deflation) a version which allows to estimate more than just the single largest element of the matrix $S$. Without going into details of the derivation, we note that~\cite[Algorithm~5.2]{HighamRelton2016} approximates the $p$ largest elements of $S$ at a cost of $2\alpha p$ matrix vector products per iteration (half of them with $S$ and half of them with $S^T$), where $\alpha \in \mathbb{N}$ is a moderate constant (typically, $\alpha = 3$ suffices). While the number of matrix-vector products per iteration of the algorithm might seem high, theoretical results and extensive numerical evidence show that the algorithm typically converges within just two iterations; see~\cite[Sections~3 and~6]{HighamRelton2016}. We note that the algorithm might fail, although this is rarely encountered in practice, barring some academic example matrices.

Returning to our setting, assume we want to find the $p$ edges with respect to which the network is most sensitive. From the Krylov subspace method outlined in Section~\ref{subsec:krylov}, we obtain the factor matrices $V_m,W_m,X_m$ at a cost of $\mathcal{O}(nm^2)$ operations. Given these matrices, matrix-vector products with $L_m \approx L_{\exp}(A_G^T, \vone\vone^T)$ can be efficiently carried out in factored form,
\begin{equation*}
    L_m\vx = V_m(X_m(W_m^T\vx)),
\end{equation*}
requiring $2nm + m^2$ arithmetic operations. Thus, subsequently applying~\cite[Algorithm~5.2]{HighamRelton2016} given the factored matrices (assuming two iterations are required for convergence) will require an overall computational cost of $4\alpha p (2nm + m^2) = 8\alpha p nm + 4\alpha p m^2$. For many real-world networks, a small number $m = \mathcal{O}(1)$ of Krylov steps is sufficient, in particular as only rough estimates of the actual sensitivities are required, as long as their relative ordering is captured accurately. In this case, the complexity of the Krylov method for approximating the Fr\'echet derivative is $\mathcal{O}(n)$ and the subsequent estimation of the largest $p$ elements has asymptotic cost $\mathcal{O}(\alpha p n)$, so that the cost of the overall method scales linearly in the number $n$ of nodes in $G$, making it feasible also for very large-scale problems, as long as $p \in \mathcal{O}(1)$.

The approach outlined above allows to estimate the $p$ largest sensitivities $S_{ij}^{\TN}$ at cost that is linear in $n$. Whether the entries that the algorithm returns belong to \emph{existing} edges $\{i,j\}\in \E$ or to \emph{``virtual''} edges $\{i,j\} \notin \E$ is outside of the control of the user. If one is interested in updating the network by adding virtual edges such that the communicability increases but the algorithm only returns sensitivities of edges already present in the network, nothing is gained. A simple, heuristic safety measure would be to estimate the $q > p$ top sensitivities and check a posteriori which of those correspond to virtual edges. This is not satisfactory for several reasons. In particular, it is not clear how much larger than $p$ the value $q$ must be chosen (and this is highly problem dependent), and additionally, the cost of algorithms scales with the number of sensitivities that one estimates. Another problem is that, occasionally, ``diagonal'' sensitivities $S_{ii}$ might be returned by the algorithm, although one will typically not want to introduce self-loops.

In order to resolve these problems, instead of applying the maximum element estimator to $L_m$, one wants to apply it to a ``masked'' version of the matrix, $L_m^{\text{masked}} := M \odot L_m$, where $M$ is a binary mask that marks candidate edges and $\odot$ denotes the Hadamard (or element-wise) matrix product. Typical choices for the binary mask are given either by (the unweighted version of) the adjacency matrix $A_G$ if only sensitivities of existing edges are required, or by an ``inverted'' version of $A_G$ (with zero diagonal to prevent self-loops), if only sensitivities of virtual edges are required.

This approach requires forming matrix vector products with $L_m^{\text{masked}}$ instead of $L_m$, which complicates the computation due to the presence of the Hadamard product. It is well-known that if one of the factors in the Hadamard product is a low-rank matrix $BC^T$ with thin $B, C \in \R^{n \times r}$, $r \ll n$, then an efficient matrix vector product is possible via
\begin{equation}\label{eq:hadamard_matvec}
    (A \odot BC^T)\vx = \sum_{i=1}^r D_{\vb_i}AD_{\vc_i}\vx,
\end{equation}
where $\vb_i, \vc_i, i = 1,\dots,r$ are the columns of $B$ and $C$, respectively, and $D_\vy$ is a diagonal matrix with the entries of the vector $\vy$ on the diagonal. The right-hand side of~\eqref{eq:hadamard_matvec} can be evaluated essentially at a cost of $r$ matrix vector products with $A$. This approach thus gives rise to an efficient matrix vector under the two conditions that $r$ is small and that $A$ exhibits a fast matrix-vector product.

When trying to estimate the largest entries in $L_m^{\text{masked}}$, we are exactly in such a situation, as $L_m$ is of rank $m \ll n$, where we typically even have $m = \mathcal{O}(1)$. Therefore, the matrix-vector product with the mask $M$ scales linearly with $n$ for ``typical'' masks: If we are interested in all existing edges, then $M = A_G$ and a matrix vector product has cost $\mathcal{O}(n)$, as $G$ is a sparse graph by assumption. If we are interested in all virtual edges, then $M = \vone\vone^T-(A_G+I)$, with which we can efficiently compute matrix vector products via $\vx \mapsto (\vone^T\vx)\vone - A_G\vx-\vx$, also at cost linear in $n$. We summarize the final procedure in Algorithm~\ref{alg:top_p}.

\begin{algorithm}
\caption{Estimating the top $p$ edge sensitivities \label{alg:top_p}}
\begin{algorithmic}[1]
\setstretch{1.2}

\smallskip

\Statex \textbf{Input:} $A_G \in \R^{n \times n}$, $\texttt{virtual} \in \{\texttt{true},\texttt{false}\}$
\Statex \textbf{Output:} $p$ [existing/virtual] edges in $G$ with highest total sensitivity

\smallskip

\State Compute $V_m, X_m, W_m$ by Algorithm~\ref{alg:krylov} with $\vb = \vc = \vone$
\State Compute singular value decomposition $U_X\Sigma_X V_X^T = X_m$
\State Set $B \leftarrow V_mU_X\Sigma_X^{1/2}$
\State Set $C \leftarrow W_mV_X\Sigma_X^{1/2}$
\If{\texttt{virtual} = \texttt{false}}
\State Set $M \leftarrow A_G$
\Else
\State Set $M \leftarrow \vone\vone^T-(A_G+I)$
\EndIf
\State Estimate top $p$ sensitivities using~\cite[Algorithm~5.2]{HighamRelton2016} for the matrix $M \odot BC^T$
\end{algorithmic}
\end{algorithm}

\begin{remark}\label{rem:top_few_existing}
Let us note that when we are interested in the top $p$ sensitivities of existing edges, it will often be preferable to simply evaluate all those sensitivities via Proposition~\ref{prop:individual_entry} at a cost of $\mathcal{O}(nm^2)$. This is the same asymptotic cost as that of the Hadamard masking approach, but the constant hidden in the $\mathcal{O}$ will typically be much larger for the latter approach. Additionally, the approach based on~\cite[Algorithm~5.2]{HighamRelton2016} might fail in rare situations.
\end{remark}

\begin{remark}\label{rem:edge_sensitivity_undirected}
When $G$ is undirected, effort can be saved by including only the upper (or only the lower) triangle of $A_G$ in the definition of the binary mask, as the sensitivity with respect to changes in $(i,j)$ is the same as the sensitivity with respect to changes in $(j,i)$.
\end{remark}

\subsection{Numerical experiments}
In this section, we perform numerical experiments to illustrate the performance of Algorithm~\ref{alg:top_p}. All experiments are carried out in MATLAB R2022a on a PC with an AMD Ryzen 7 3700X 8-core CPU with clock rate 3.60GHz and 32 GB RAM.

\begin{example}
\begin{figure}
    \centering
    \includegraphics[width=.99\textwidth]{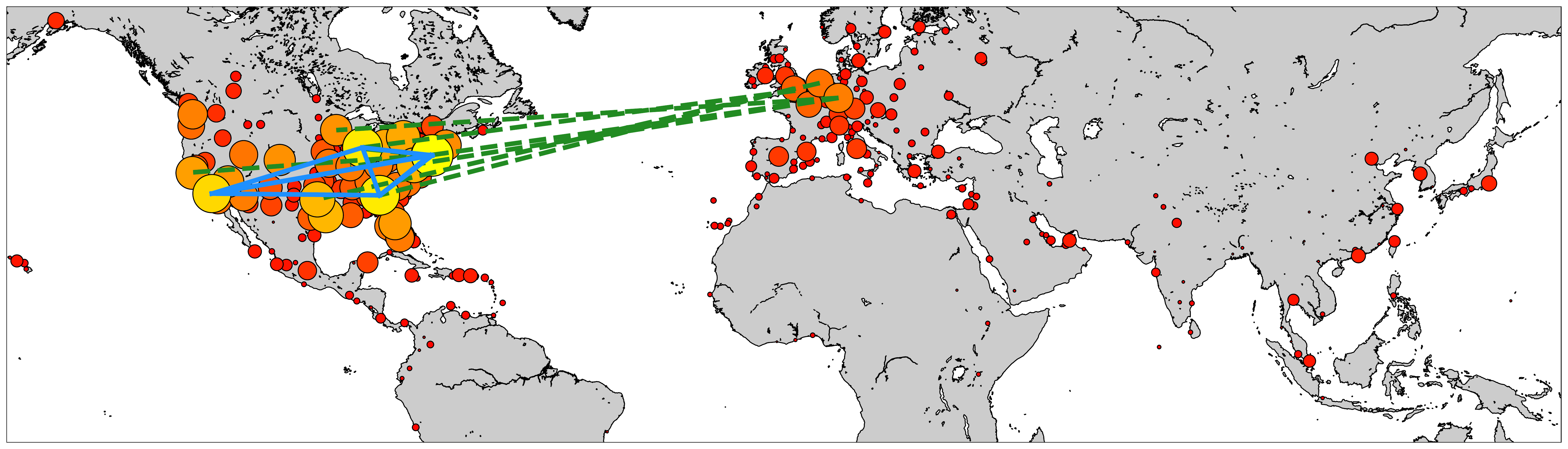}
    \caption{Visualization of the \texttt{Air500} network. The blue, solid edges are existing edges with highest sensitivity $S^{\TN}_{ij}$, while the green, dashed edges are \emph{non-existing/virtual} edges with highest sensitivity $S^{\TN}_{ij}$ (according to Algorithm~\ref{alg:top_p}); see the text for details. All other edges are omitted. The size and color of nodes encode their total communicability (with lighter colors corresponding to larger values). Note that a few nodes (with low communicability) of the network are not included in the excerpt, as they lie farther to the north or south (world map generated with the Python \texttt{basemap} package).}\label{fig:air500}
\end{figure}
Our first example is inspired by~\cite[Example~5.2]{DelaCruzCabreraJinNoscheseReichel2021} and uses the network \texttt{Air500}~\cite{MarcelinoKaiser2012}, which contains the top $n = 500$ airports in the world as nodes (based on passenger volume between July 2007 and June 2008) and models flights between these airports as edges (which gives $|\E| = 24009$ edges in total). The graph is directed and unweighted. Figure~\ref{fig:air500} contains a visualization of a large part of the network (in which we left out most edges for improved clarity). The color and size of nodes encode their total communicability~\eqref{eq:total_communicability}.

\begin{table}
\caption{Run time and number of Krylov iterations for the compared algorithms applied to the \texttt{Air500} network. For the algorithm from~\cite{DelaCruzCabreraJinNoscheseReichel2021}, we report the \emph{average} number of Krylov iterations across all calls to Algorithm~\ref{alg:krylov}.}
\centering
\renewcommand{\arraystretch}{1.25}
\begin{tabular}{l|l|r|r|r}
\hline\hline
 & \textbf{method} & \textbf{run time} & \textbf{Krylov it.} & calls of Alg.~\ref{alg:krylov}  \\
\hline
\multirow{2}{*}{\textbf{existing edges}}  & Algorithm~\ref{alg:top_p} & 0.1s & 11 & 1 \\
\cline{2-5}
                                 & Method from~\cite{DelaCruzCabreraJinNoscheseReichel2021} & 50.5 s & 9.4 & 24009 \\
\hline
\multirow{2}{*}{\textbf{virtual edges}}   & Algorithm~\ref{alg:top_p} & 0.1 s & 11 & 1 \\
\cline{2-5}
                                 & Method from~\cite{DelaCruzCabreraJinNoscheseReichel2021} & 482.2 s & 9.73 & 225491 \\
\hline
\hline
\end{tabular}
    \label{tab:air500_time}
\end{table}

\begin{table}
\caption{Existing and virtual edges with highest sensitivity according to our method and the method from~\cite{DelaCruzCabreraJinNoscheseReichel2021} (with highest sensitivity at the top). Edges are identified with flight connections, using the three-character IATA codes of the corresponding airports. For existing edges, both methods yield exactly the same result, while for virtual edges they differ slightly (edges that are selected by just one of the algorithms are marked in bold). Updating the graph using the edges selected by the method from~\cite{DelaCruzCabreraJinNoscheseReichel2021} increases the total communicability of the network by $13.21\%$ while using the edges selected by our new method increases it by $12.19\%$.}
\centering
\renewcommand{\arraystretch}{1.25}
\begin{tabular}{c|c||c|c}
\hline\hline
\multicolumn{2}{c||}{\textbf{existing edges}} & \multicolumn{2}{c}{\textbf{virtual edges}} \\
\hline
\textbf{Algorithm~\ref{alg:top_p}} & \textbf{Method from~\cite{DelaCruzCabreraJinNoscheseReichel2021}} & \textbf{Algorithm~\ref{alg:top_p}} & \textbf{Method from~\cite{DelaCruzCabreraJinNoscheseReichel2021}}\\
\hline
JFK -- ATL & JFK -- ATL & JFK -- LGA & JFK -- LGA\\  
ORD -- JFK & ORD -- JFK & LHR -- ATL & \textbf{LGA -- JFK}\\
JFK -- ORD & JFK -- ORD & AMS -- DFW & LHR -- ATL\\
ATL -- JFK & ATL -- JFK & JFK -- MDW & AMS -- DFW\\
JFK -- LAX & JFK -- LAX & \textbf{ORD -- MDW} & ATL -- LHR\\
EWR -- JFK & EWR -- JFK & \textbf{FRA -- MSP} & \textbf{MDW -- JFK}\\ 
JFK -- EWR & JFK -- EWR & \textbf{LGW -- ORD} & \textbf{JFK -- MDW}\\
ORD -- ATL & ORD -- ATL & \textbf{FRA -- BWI} & \textbf{ABQ -- JFK}\\
LAX -- JFK & LAX -- JFK & \textbf{OAK -- EWR} & \textbf{DFW -- AMS}\\
ATL -- ORD & ATL -- ORD & \textbf{FRA -- STL} & \textbf{ORD -- LGW}\\
\hline\hline
\end{tabular}
\label{tab:air500_edges}
\end{table}

Similar to what was done in~\cite[Example~5.2]{DelaCruzCabreraJinNoscheseReichel2021}, we try to find the top $p = 10$ existing and the top $p = 10$ virtual edges in the network according to the sensitivity of total communicability. We compare our method, Algorithm~\ref{alg:top_p}, to the basic Krylov method used in~\cite{DelaCruzCabreraJinNoscheseReichel2021}, which essentially computes all individual sensitivities by evaluating one Fr\'echet derivative per edge and then selects the edges with the $p$ largest values. To make comparisons as fair as possible, we also use Algorithm~\ref{alg:krylov} as backbone for this method (in~\cite{DelaCruzCabreraJinNoscheseReichel2021}, several different Krylov methods were introduced, but the method from~\cite{KandolfKoskelaReltonSchweitzer2021,Kressner2019} turned out to be among those giving the best balance between speed and accuracy). As stopping criterion for the Krylov method, we compute the norm of the difference between consecutive iterates (which can be computed without explicitly forming the iterates; cf.~\cite[Section~5]{KandolfKoskelaReltonSchweitzer2021}) and check whether it is below a prescribed tolerance $\texttt{tol}$. Note that in our method, the accuracy requirement applies to the matrix containing all sensitivities, while in the method from~\cite{DelaCruzCabreraJinNoscheseReichel2021}, it is applied to each individual sensitivity. Thus, to obtain a fair comparison, we reduce the tolerance to $\texttt{tol}/n$ in our method. As a rather crude accuracy is typically sufficient (as we are mainly interested in the ranking of the nodes, not the precise sensitivity values), we use $\texttt{tol} = 10^{-3}$ in this experiment. For the maximum element estimator, we choose $\alpha = 3$.

The run time and number of Krylov iterations required by the different methods are depicted in Table~\ref{tab:air500_time} and the edges that both algorithms select are listed in Table~\ref{tab:air500_edges}. Note that the method from~\cite{DelaCruzCabreraJinNoscheseReichel2021} requires one call to Algorithm~\ref{alg:krylov} for each existing/virtual edge, and we report the average number of Krylov iterations over all these calls in Table~\ref{tab:air500_time}. 

As one would expect, both methods yield very similar results (for existing edges, both results are actually identical). In case of deviations, one can expect the results of the method from~\cite{DelaCruzCabreraJinNoscheseReichel2021} to be closer to the ``ground truth ranking'', as all sensitivities are explicitly computed. In Algorithm~\ref{alg:top_p}, it might happen that a few top edges are missed by the maximum element estimator, in particular if sensitivity scores of multiple edges are very close to each other (as it is the case here). Still, in this setting, the update suggested by our method will also be sensible from an application point of view. To confirm this, we compute the actual effect that both updates have on the total communicability of the network. Introducing the edges selected by the method from~\cite{DelaCruzCabreraJinNoscheseReichel2021} increases the total communicability of the network by $13.21\%$, while the update computed by Algorithm~\ref{alg:top_p} increases it by $12.19\%$. Thus, we find an update which is almost as good as the ``ground truth update'', but at a cost which is several orders of magnitude smaller.

It is interesting to note that the top $p=10$ existing edges connect five large US airports (John F.\ Kennedy, Newark, Chicago O'Hare, Atlanta, Los Angeles), while many of the virtual edges that our method selects for greatly improving the total communicability of the network connect large European airports (Amsterdam, Frankfurt, London Heathrow, London Gatwick) to US airports.

The run time of Algorithm~\ref{alg:top_p} is much lower than that of the method from~\cite{DelaCruzCabreraJinNoscheseReichel2021}, as we only need to approximate one Fr\'echet derivative, instead of $24009$ (existing edges) or $500\cdot499 - 24009 = 225491$ (virtual edges). To reduce run time of the method from~\cite{DelaCruzCabreraJinNoscheseReichel2021}, one could of course only compute sensitivities of edges between nodes with high total communicability (e.g., the top 10\%, similar to what is done in up/downdating heuristics for large scale networks in~\cite{ArrigoBenzi2016a,ArrigoBenzi2016b}), as it is very likely that the edges with highest sensitivity belong to this set. Even then, the run time of our new method can be expected to still be orders of magnitude smaller.

We note that it is crucial for the efficient applicability of our method that the number of required Krylov iterations is quite small (and independent of the network size $n$, if we want to obtain linear scaling), as it also directly influences the cost of the second stage of the method, as matrix-vector products with $V_mX_mW_m^T$ become more expensive as $m$ grows; see also the discussion in Section~\ref{subsec:top_sensitivities}. It is observable from the results in Table~\ref{tab:air500_time}, that a small number of iterations, $m = 11$, is sufficient to reach the desired accuracy $2\cdot 10^{-6}$ for this example network. 
\end{example}

\begin{example}\label{ex:geom}
\begin{figure}
    \centering
    \includegraphics[width=.45\textwidth]{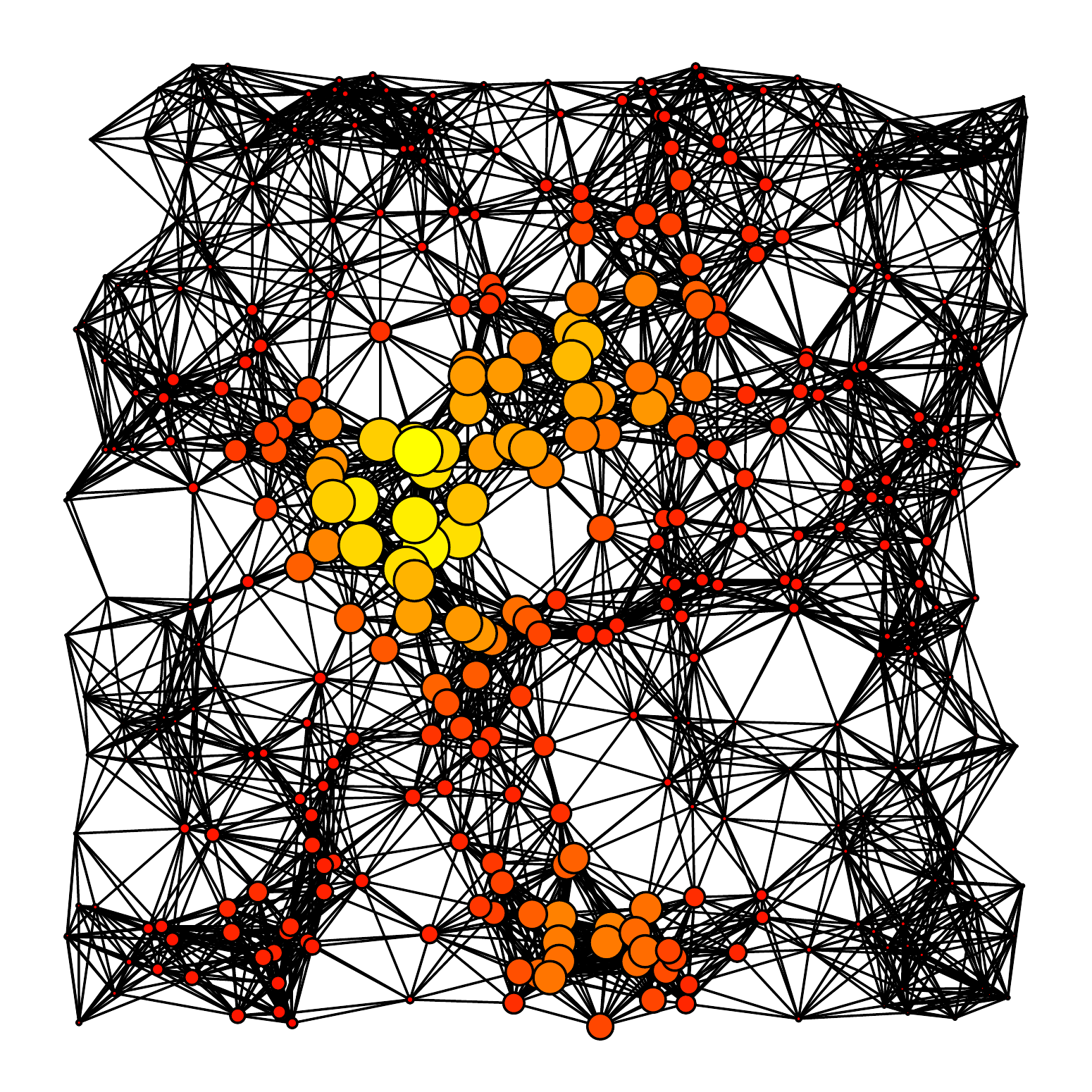}
    \caption{Illustration of random geometric graph with $n = 400$ nodes. The size and color of nodes encodes their total communicability (with lighter colors corresponding to larger values).}
    \label{fig:geom_graph}
\end{figure}
\begin{table}
\caption{Results obtained for random geometric graphs of varying size. ``Kryl.~it.'' refers to number of iterations in Algorithm~\ref{alg:krylov}, while ''HR it.'' refers to number of iterations in~\cite[Algorithm~5.2]{HighamRelton2016}. Entries marked with * indicate that the method did not finish running within a limit of two hours.}
    \label{tab:geom}
\centering
\renewcommand{\arraystretch}{1.5}
\begin{tabular}{l|l|l|l|l|l|l|l|l}
\hline\hline
& \ \ \ \ \ \ \ $n$ & $200$ & $400$ & $800$ & $1600$ & $3200$ & $6400$ & $12800$ \\
\cline{2-9}
& avg.\ deg.\ & 9.88 & 10.0 & 10.4& 10.9 & 11.0 & 11.2 & 11.2 \\
\hline
\multirow{3}{*}{Alg.~\ref{alg:top_p}} & Kryl.~it.\ & 14 & 17 & 18 & 22 & 22 & 23 & 25\\
& HR it.\     & 4 & 2 & 4 & 2 & 3 & 5 & 2 \\
& time    & 0.02s & 0.03s & 0.06s & 0.08s & 0.26s & 0.70s & 0.73s \\
\hline
\multirow{2}{*}{Alg.~from~\cite{DelaCruzCabreraJinNoscheseReichel2021}}  & Kryl.~it.\  & 12.7 & 14.8 & 15.7 & 14.9 & * & * & *\\
& time    & 40s & 216s & 1049s & 5417s & * & * & * \\
\hline
\hline
\end{tabular}
\end{table}
We now perform an experiment in which we use artificially constructed graphs in order to illustrate the scaling behavior of Algorithm~\ref{alg:top_p}. Specifically, we construct a \emph{random geometric graph} by sampling $n$ uniformly distributed points in the unit square and then connecting all pairs with distance below some threshold $d$ by an edge. We vary the size of the graph from $n = 200$ to $n = 12800$ and choose the distance threshold $d$ in dependence on $n$ such that the average degree in the resulting graph is roughly 10 (such that it is sensible to consider all graphs as different-sized instances of the same problem). As an example, the graph resulting for $n = 400$ is depicted in Figure~\ref{fig:geom_graph}. We use the same Krylov accuracies and the same value of $\alpha$ as in the previous experiment. Detailed results are reported in Table~\ref{tab:geom}. We observe that the number of Krylov iterations necessary to satisfy the tolerance requirement slightly increases when increasing the problem size. As expected, the number of iterations in the maximum element estimator is consistently small, in line with theoretical and numerical evidence reported in~\cite{HighamRelton2016}. Precisely, it ranges from 2 to 5, with no clear dependence on the matrix size. Concerning execution times, it is clearly visible that Algorithm~\ref{alg:top_p} indeed scales linearly in the problem size, while the cubic scaling of the method from~\cite{DelaCruzCabreraJinNoscheseReichel2021} leads to enormous run times which exceed two hours for the problem of size $n = 3200$, while Algorithm~\ref{alg:top_p} stays below one second also for the largest problem instance with $n = 12800$.
\end{example}

\begin{example}
\begin{table}
\caption{Number of nodes, number of edges and total communicability of test networks from the SuiteSparse collection.}
    \centering
    \renewcommand{\arraystretch}{1.25}
    \begin{tabular}{l|c|c|c}
    \hline\hline
         \textbf{Network} & $n$ & $|\E|$ & $C^{\TN}$  \\
         \hline
         Pajek/Erdos972 & 5488 & 14170 & $8.4 \cdot 10^8$ \\
         \hline
         Pajek/Erdos982 & 5822 & 14750 & $1.2 \cdot 10^9$ \\
         \hline
         Pajek/Erdos992 & 6100 & 15030 & $1.5 \cdot 10^9$ \\
         \hline
         SNAP/ca-GrQc   & 5242 & 28980 & $4.6 \cdot 10^{21}$ \\
         \hline
         SNAP/ca-HepTh  & 9877 & 51971 & $1.0 \cdot 10^{15}$ \\ 
         \hline
         SNAP/as-735    & 7716 & 26467 & $2.8 \cdot 10^{23}$ \\
    \hline\hline
    \end{tabular}
    \label{tab:suitesparse}
\end{table}
\begin{table}[t]
    \caption{Results obtained for test networks from the SuiteSparse collection (see Table~\ref{tab:suitesparse} for details on their properties). ``Kryl.~it.'' refers to number of iterations in Algorithm~\ref{alg:krylov}, while ''HR it.'' refers to number of iterations in~\cite[Algorithm~5.2]{HighamRelton2016}.}

    \centering
    \renewcommand{\arraystretch}{1.25}
    \begin{tabular}{l|c|c|c|c|c}
    \hline\hline
         \textbf{Network} & $p$ & Kryl.\ it.\ & HR.\ it.\ & time & incr.\ of $C^{\TN}$ \\
         \hline
         \multirow{3}{*}{Pajek/Erdos972} & $10$   & \multirow{3}{*}{$18$} & 3 & 0.21s & 45\%   \\
                                         & $50$   &                       & 3 & 0.78s & 488\%  \\
                                         & $100$  &                       & 3 & 1.66s & 4026\% \\
         \hline
         \multirow{3}{*}{Pajek/Erdos982} & $10$   & \multirow{3}{*}{$18$} & 3 & 0.18s & 43\%   \\
                                         & $50$   &                       & 3 & 0.97s & 408\%  \\         
                                         & $100$  &                       & 3 & 2.08s & 1981\% \\
         \hline
         \multirow{3}{*}{Pajek/Erdos992} & $10$   & \multirow{3}{*}{$18$} & 2 & 0.16s & 43\%   \\
                                         & $50$   &                       & 3 & 0.89s & 470\%  \\         
                                         & $100$  &                       & 3 & 1.82s & 2612\% \\
         \hline
         \multirow{3}{*}{SNAP/ca-GrQc}   & $10$   & \multirow{3}{*}{$16$} & 2 & 0.14s & 33\%   \\
                                         & $50$   &                       & 4 & 1.21s & 95\%   \\         
                                         & $100$  &                       & 2 & 1.37s & 288\%  \\
         \hline
         \multirow{3}{*}{SNAP/ca-HepTh}  & $10$   & \multirow{3}{*}{$18$} & 2 & 0.28s & 26\%    \\
                                         & $50$   &                       & 3 & 1.32s & 532\%   \\         
                                         & $100$  &                       & 2 & 2.00s & 941\%  \\
         \hline
         \multirow{3}{*}{SNAP/as-735}    & $10$   & \multirow{3}{*}{$15$} & 3 & 0.24s & 13\%   \\
                                         & $50$   &                       & 3 & 1.19s & 93\%   \\         
                                         & $100$  &                       & 4 & 3.43s & 224\%  \\
    \hline\hline
    \end{tabular}
    \label{tab:suitesparse_results}
\end{table}
In a last example, we demonstrate the performance of our method on several real-world networks from the SuiteSparse matrix collection (\url{https://sparse.tamu.edu/}) which are frequently used when investigating total communicability; see, e.g.,~\cite{BenziBoito2020,BenziKlymko2013}. We summarize the most important properties of the data set in Table~\ref{tab:suitesparse}. In this experiment, we do not perform a comparison to the original method from~\cite{DelaCruzCabreraJinNoscheseReichel2021}, as the results of the previous experiment already clearly indicate that run times would be extremely high for the network sizes under consideration.

For each of the networks, we perform updates by introducing $p = 10, 50, 100$ virtual edges and report the iteration numbers, run times and the increase of total communicability that is achieved by the update. The parameters of the method are again chosen as in the previous experiments.
The results are given in Table~\ref{tab:suitesparse_results}. We can observe that (as expected) the execution time of the method also scales almost perfectly linear with $p$, with a few exceptions in those cases where a larger number of iterations of Algorithm~\ref{alg:power} is needed for some values of $p$ (e.g., for the SNAP/ca-GrQc network and $p=50$). In all cases, the introduced updates clearly benefit the total communicability substantially, although the ratios by which it increases greatly vary. In particular for the three Erd\H{o}s collaboration networks, enormous increases are achieved for $p = 100$. Still, even the worst result for $p=100$ (the SNAP/as-735 network) more than doubles the total communicability, although the number of newly introduced edges is less than $1\%$ of the number of existing edges. Thus, while we cannot give precise guarantees for how well our updates approximate the ``best'' update with $p$ edges, the results clearly indicate that very good updates are produced also for larger real-world networks.
\end{example}

\section{Some extensions of the network sensitivity concept}\label{sec:sensitivity}
In this section, we first show how the sensitivity concept from Definition~\ref{def:sensitivity} can be extended to subgraph centrality and the Estrada index, and then we brief\/ly discuss how sensitivity with respect to removal (or outage) of certain nodes can be incorporated into the framework. For all considered cases, computational procedures similar to the one introduced in Section~\ref{sec:algorithms} can be derived in a straightforward fashion. We therefore do not go into detail concerning this topic.

\subsection{Sensitivity of subgraph centrality and the Estrada index}\label{subsec:sensitivity_subgraph_centrality} 
The concept of total network sensitivity from~\cite{DelaCruzCabreraJinNoscheseReichel2021} can straightforwardly be extended to the influence of edge modifications on subgraph centrality~\eqref{eq:exponential_subgraph_centrality} and the Estrada index~\eqref{eq:estrada_index} instead of total communicability, yielding the following analogue of Definition~\ref{def:sensitivity}. Sensitivity of subgraph centrality is especially relevant if one wants to judge the influence of network modifications on the centrality of a particular node instead of the communicability of the network as a whole.
\begin{definition}\label{def:sensitivity_subgraph_estrada}
Let $G = (\V, \E, w)$ be a weighted (di)graph with adjacency matrix $A_G \in \Rnn$ and let $E_{ij} := \ve_i\ve_j^T \in \Rnn$. Then, the \emph{sensitivity of subgraph centrality of node $v$} with respect to changes in $w_{ij}$ is defined as
\begin{equation}\label{eq:sensitivity_subgraph}
S^{\SC}_{ij}(v) := \ve_v^TL_{\exp}(A_G, E_{ij})\ve_v
\end{equation}
and the \emph{sensitivity of the Estrada index} with respect to changes in $w_{ij}$ is defined as
\begin{equation}\label{eq:sensitivity_estrada_index}
S^{\EE}_{ij}(A_G) := \trace \left( L_{\exp}(A_G, E_{ij})\right).
\end{equation}
\end{definition}

\begin{remark}\label{rem:estrada_rank1}
Note that of course~\eqref{eq:transposition_invariance} also holds with $\vone$ replaced by $\ve_v$, so that again, for undirected graphs, it is sensible to define sensitivity of subgraph centrality using just rank-one terms.
\end{remark}

The following elementary result (which follows directly from the multivariate chain rule) shows that the quantities defined in Definition~\ref{def:sensitivity_subgraph_estrada} do indeed measure sensitivity of the respective network indices.
\begin{proposition}\label{prop:sensitivity_frechet}
The sensitivity of subgraph centrality and the Estrada index defined in Definition~\ref{def:sensitivity_subgraph_estrada} are the rates of change of the respective quantities w.r.t.\ changes in the weight $w_{ij}$, i.e.,\
\begin{equation*}
    \frac{\partial}{\partial{}w_{ij}} \vc^{\SC}(v; w) = S^{\SC}_{ij}(v) \qquad\text{ and }\qquad \frac{\partial}{\partial{}w_{ij}} EE(A_G; w) = S^{\EE}_{ij}(A_G),
\end{equation*}
where the notations $\vc^{\SC}(\ \cdot\ ; w)$ and $EE(\ \cdot\ ; w)$ are meant to explicitly show the dependence of the corresponding network indices on the underlying weight function.
\end{proposition}

Next, we give an alternative characterization of the sensitivities introduced above. We again start by introducing a rather general result and then state a corollary for our specific setting, which is similar in spirit to Corollary~\ref{cor:total_sensitivity_simple_formula} for total sensitivity.

\begin{theorem}\label{the:frechet_trace}
Let $A \in \Cnn$ and let $f$ be analytic on a region that contains $\spec(A)$. Then
\begin{equation*}
    \trace(L_{f}(A,E_{ij})) = [f^\prime(A^T)]_{ij}.
\end{equation*}
\end{theorem}
\begin{proof}
First note that due to the definition of the trace, Theorem~\ref{the:frechet_eij} and the linearity of the Fr\'echet derivative, we have
\begin{equation}\label{eq:proof_trace1}
\trace(L_{f}(A,E_{ij})) = \sum_{v=1}^n \ve_v^T[L_{f}(A^T,E_{ij})]_{ij}\ve_v = [L_{f}(A^T,\sum_{v=1}^n \ve_v\ve_v^T)]_{ij} = [L_{f}(A^T,I_n)]_{ij},
\end{equation}
where $I_n$ denotes the identity matrix of size $n \times n$. Now, because $f$ is analytic on a region containing $\spec(A)$, we can use the integral formula~\eqref{eq:frechet_derivative_integral} for the Fr\'echet derivative, which gives
\begin{equation}\label{eq:proof_trace2}
    L_f(A^T,I_n) = \frac{1}{2\pi i} \int_\Gamma f(\zeta)(\zeta I - A^T)^{-1}I_n(\zeta I - A^T)^{-1}\d \zeta = \frac{1}{2\pi i} \int_\Gamma f(\zeta)(\zeta I - A^T)^{-2}\d \zeta.
\end{equation}
The right-hand side of~\eqref{eq:proof_trace2} corresponds to the Cauchy integral formula for the derivative of $f$, so that we find
\begin{equation}\label{eq:proof_trace3}
L_f(A^T,I_n) = \frac{1}{2\pi i} \int_\Gamma f(\zeta)(\zeta I - A^T)^{-2}\d \zeta = f^\prime(A^T).
\end{equation}
Inserting~\eqref{eq:proof_trace3} into~\eqref{eq:proof_trace1} concludes the proof.
\end{proof}

\begin{corollary}\label{cor:alternative_characterization_subgraph}
Let $S^{\SC}_{ij}(v)$ and $S^{\EE}_{ij}(A_G)$ denote the sensitivities of subgraph centrality and the Estrada index defined in~\eqref{eq:sensitivity_subgraph}--\eqref{eq:sensitivity_estrada_index}. Then
\begin{equation}\label{eq:alternative_subgraph}
S^{\SC}_{ij}(v) = [L_{\exp}(A_G^T,\ve_v\ve_v^T)]_{ij}.
\end{equation}
and
\begin{equation}\label{eq:alternative_estrada}
    S^{\EE}_{ij}(A_G) = [\exp(A_G^T)]_{ij}.
\end{equation}
\end{corollary}
\begin{proof}
Relation~\eqref{eq:alternative_subgraph} follows by applying Theorem~\ref{the:frechet_eij} to $f(A) = \exp(A_G)$ and $\vu = \vv = \ve_v$. Similarly,~\eqref{eq:alternative_estrada} follows by applying Theorem~\ref{the:frechet_trace} to $\trace(L_{\exp}(A_G,E_{ij})$, noting that $f = f'$ for $f = \exp$.
\end{proof}

\begin{remark}\label{rem:alternative_estrada}
We brief\/ly comment on formula~\eqref{eq:alternative_estrada} for the sensitivity of the Estrada index, as it reveals a quite curious connection. The entry $[\exp(A_G^T)]_{ij} = [\exp(A_G)]_{ji}$ is determined by the number and lengths of walks in $G$ that start at node $j$ and end at $i$ (for an undirected graph, this quantity was introduced as \emph{communicability} of nodes $i$ and $j$ in~\cite{EstradaHatano2008}). It is quite interesting that this number single-handedly controls how sensitive the Estrada index reacts to modifications in the edge $(i,j)$.
\end{remark}
\begin{figure}
    \centering
    \includegraphics[width=.9\textwidth]{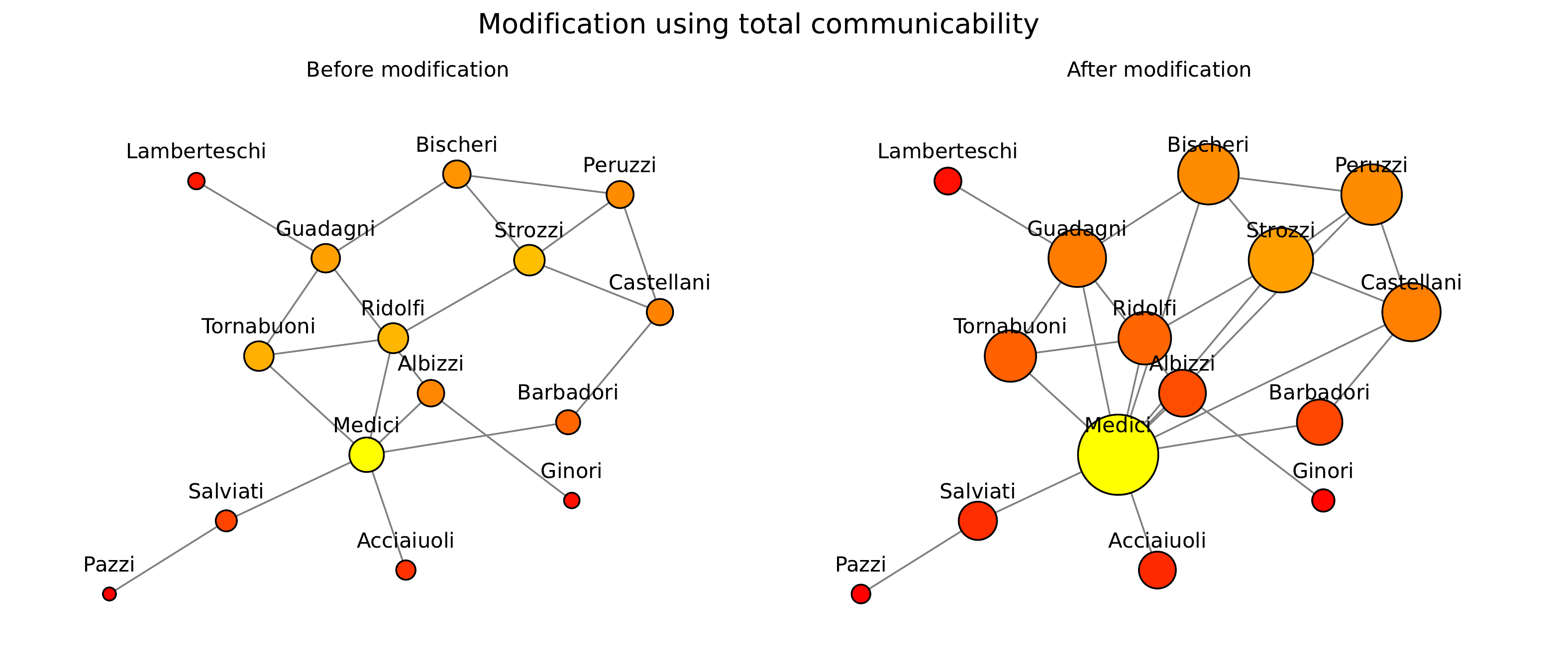}
    
    \vspace{.6cm}
    \includegraphics[width=.9\textwidth]{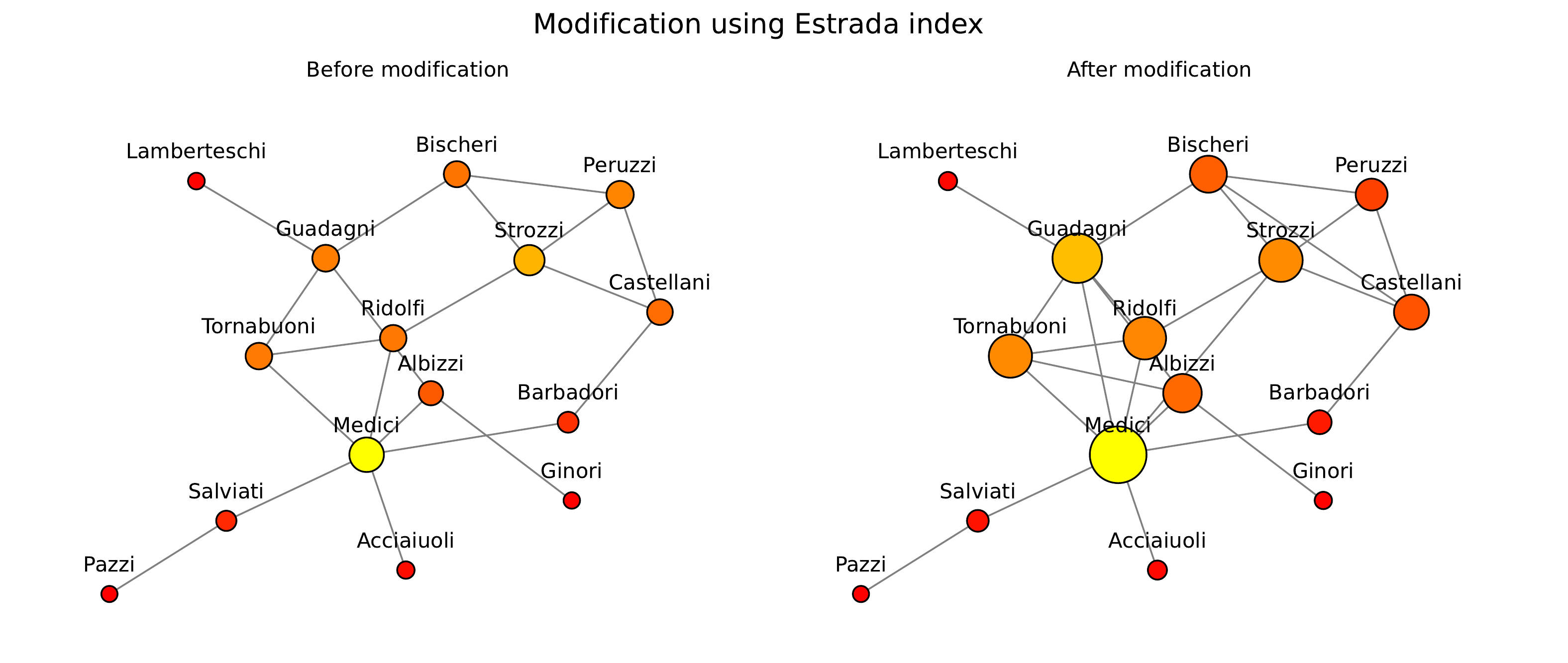}
    \caption{Florentine family network before and after edge modification. \emph{Top row:} Edge modification according to total sensitivity $S_{ij}^{\tn}$. The size and coloring of the nodes indicate the total communicability of each node before (left) and after the modification (right), with lighter colors indicating higher values. \emph{Bottom row:} Edge modification according to sensitivity of the Estrada index $S_{ij}^{\EE}$. The size and coloring of the nodes indicate the subgraph centrality of each node before (left) and after the modification (right). Note that we deliberately do not use the same color scale in the left and right plots, but instead have the color be determined by the \emph{relative} importance of the node in its specific network.}
    \label{fig:florentine_families}
\end{figure}
\begin{table}
\caption{Virtual edges $(i,j)$ with highest total sensitivity $S_{ij}^{\tn}$ and highest sensitivity of the Estrada index $S_{ij}^{\EE}$ for the Florentine family network.}
\centering
\renewcommand{\arraystretch}{1.25}
\setlength{\tabcolsep}{10pt}
\begin{tabular}{l|c||l|c}
\hline\hline
\multicolumn{2}{c||}{\textbf{Total sensitivity}} & \multicolumn{2}{c}{\textbf{Estrada index}} \\
\hline
Edge & $S_{ij}^{\tn}$  & Edge & $S_{ij}^{\EE}$\\
\hline
    Medici -- Strozzi & 42.22 & Medici -- Guadagni & 2.73 \\
    Medici -- Guadagni & 39.40 & Bischeri -- Castellani & 2.46 \\
    Medici -- Bischeri & 36.20 & Tornabuoni -- Albizzi & 2.36 \\
    Medici -- Peruzzi & 35.33 & Medici -- Strozzi & 2.10 \\
    Medici -- Castellani & 34.26 & Guadagni -- Ridolfi & 2.02\\
\hline\hline
\end{tabular}
\label{tab:florentine_families}
\end{table}  
We now illustrate on a small example that it is indeed worthwhile to consider sensitivity of subgraph centrality instead of total sensitivity, as it can give quite different results. 

\begin{example}\label{ex:florentine_families}
We consider a simple unweighted, undirected example network taken from~\cite{BreigerPattison1986}; see also~\cite{BenziBoito2020}. The network consists of 15 nodes, representing Florentine families in the 15th century, and 20 edges, representing marriages between the families. 

Assume we want to add five edges to the network with the goal to increase its communicability as much as possible. To find suitable edges, we compute the sensitivities $S^{\tn}_{ij}$ and $S^{\EE}_{ij}$ for all virtual edges $(i,j)$ and then add the five edges with highest sensitivity to the network. Depending on whether we use total sensitivity or sensitivity of the Estrada index, different edges are selected. The results are summarized in Table~\ref{tab:florentine_families}. Interestingly, the approach based on total communicability selects five edges all involving the Medici family (which is the family with by far highest total communicability), while the approach based on the Estrada index selects only two edges involving Medici and three edges involving other families, thus leading to a more ``balanced'' update of the network. Figure~\ref{fig:florentine_families} depicts the network, the communicability scores of the nodes and the updates resulting from both approaches.

This example clearly illustrates that network modifications based on total communicability and subgraph centrality can lead to quite different results. Which result is more appropriate for efficiently updating the network at hand of course depends on the specific application. For the above example, one could argue that the ``balanced'' update obtained by using the Estrada index might be preferable to the ``Medici-focused'' update (e.g., because there might be too few Medici descendants for so many marriages).
\end{example}

\subsection{Network sensitivity with respect to node removal}\label{subsec:node_removal}
When investigating robustness of networks, another modification that is certainly of interest besides edge addition/removal is the \emph{removal of a node} from the network. This modification fits into our framework by modeling it as the removal of all in-/outgoing edges of the node, isolating it from the rest of the network.

For easier notation, we denote the $v$th row and column of $A_G$ by $\va_{v:} := \ve_v^TA_G$ and $\va_{:v} := A_G\ve_v$, respectively, and define 
\begin{equation}\label{eq:modification_node_removal}
A_G(h) = A_G - hE_v \quad\text{ where } E_v = -(\ve_v \va_{v:} + \va_{:v}\ve_v^T).
\end{equation}

This way, $A_G(0) = A_G$, while $A_G(1)$ is the adjacency matrix of the graph in which all edges incident to node $v$ are removed and all other edges remain unchanged. Thus, in light of~\eqref{eq:modification_node_removal}, it is natural to define the sensitivities with respect to removal of node $v$ via
\begin{align}\label{eq:network_sensitivity_nodes}
S^{\tn}_{v}(A_G) &:= \vone^T L_{\exp}(A_G, E_v) \vone, \nonumber\\
S^{\SC}_{v}(u) &:= \ve_u^T L_{\exp}(A_G, E_v) \ve_u, \\
S^{\EE}_{v}(A_G) &:= \trace(L_{\exp}(A_G, E_v)). \nonumber
\end{align}

Clearly, $E_v$ is a linear combination of those $E_{iv}$ and $E_{vj}$ for which $(i,v) \in \E$ or $(v,j) \in \E$. By the linearity of the Fr\'echet derivative, analogous relations to those in Corollary~\ref{cor:total_sensitivity_simple_formula} and Corollary~\ref{cor:alternative_characterization_subgraph} therefore also hold for the measures~\eqref{eq:network_sensitivity_nodes}, characterizing them as the (possibly weighted) sum of a few entries of the matrices in~\eqref{eq:total_sensitivity_simple_formula},~\eqref{eq:alternative_subgraph} and~\eqref{eq:alternative_estrada}, respectively. We refrain from explicitly stating all of the corresponding formulas, as they are completely straightforward. Just as an example, for sensitivity of total communicability, we find
\begin{align*}
S_v^{\TN}(A_G) &= \sum_{(i,v) \in \E} S_{i,v}^{\TN}(A_G) + \sum_{(v,j) \in \E} S_{v,j}^{\TN}(A_G) \\
               &= \sum_{(i,v) \in \E} [L_{\exp}(A_G^T,\vone\vone^T)]_{iv} + \sum_{(v,j) \in \E} [L_{\exp}(A_G^T,\vone\vone^T)]_{vj},
\end{align*}
which can also be evaluated by forming just a single Fr\'echet derivative.

\begin{remark}\label{rem:centrality_measure}
The sensitivity measures~\eqref{eq:network_sensitivity_nodes} can be used for analyzing robustness of a network with respect to outages, targeted attacks etc.\ of certain nodes. This might be particularly interesting for directed networks: in contrast to undirected networks, it is rather difficult to find a good approach for assigning a single centrality score to a node. For example, it turns out that the diagonal entries of $\exp(A_G)$ do not need to carry any meaningful information in the directed case; see, e.g.,~\cite{BenziEstradaKlymko2013}. Instead, it is more appropriate to assign two scores to a node, one which measures the importance as \emph{broadcaster} and one which measures the importance as \emph{receiver} of information. These measures give information  on the nature of information flow and the \emph{roles} of nodes in the network, but they do not straightforwardly answer questions about vulnerability of the network to outage of certain nodes. For example, it is not clear whether a network would be most strongly affected by the removal of an important broadcaster, an important receiver, or a node which is not at the top of any of the two categories, but takes on both roles reasonably well. For this question, the measures~\eqref{eq:network_sensitivity_nodes}, in particular $S_{v}^{\TN}(A_G)$, could thus potentially yield meaningful additional information. 
\end{remark}

\section{Decay bounds for the Fr\'echet derivative and a priori bounds for sensitivity to edge or node modifications}\label{sec:decay_bounds}

It is well-known that the entries of matrix functions $f(A)$ often exhibit an exponential or even super-exponential decay away from the sparsity pattern of $A$: The larger the geodesic distance $d(u,v)$ of node $u$ and $v$ in the graph of $A$, the smaller the entry $[f(A)]_{uv}$ can be expected to be. This was first studied for the inverse of banded matrices in~\cite{DemkoMossSmith1984} and later extended to other functions and matrices with more general sparsity pattern in numerous works; see, e.g.,~\cite{BenziRazouk2007, Benzi2016, BenziRinelli2022, BenziSimoncini2015, FrommerSchimmelSchweitzer2018_1, FrommerSchimmelSchweitzer2018_2, PozzaSimoncini2021, Schimmel2019, Schweitzer2022a} and the references therein; Specifically, in~\cite{PozzaTudisco2018} such an approach was applied in the context of network modifications.

\subsection{Sparsity structure of Fr\'echet derivatives}\label{subsec:sparsity_frechet}
In this section, we investigate decay properties of the Fr\'echet derivative $L_f(A, E)$, when the direction term is of the special form $E = E_{ij}$ or $E = E_v$ considered in this work.

In order to apply techniques similar to those often used for proving decay in $f(A)$, we start by investigating the sparsity pattern of Fr\'echet derivatives of polynomial matrix functions. In the following, we denote the set of all polynomials of degree at most $m$ by $\Pi_m$. The following elementary result forms the basis of our derivations.
\begin{proposition}\label{prop:frechet_derivative_polynomial}
Let $A, E \in \Rnn$ and let $p_m(z) = \sum_{k=0}^m \alpha_k z^k \in \Pi_m$. Then the Fr\'echet derivative of $p_m$ at $A$ in direction $E$ is given by
\begin{equation}\label{eq:frechet_derivative_polynomial}
    L_{p_m}(A, E) = \sum\limits_{k=1}^m \alpha_k \sum\limits_{\ell=1}^k A^{\ell-1}EA^{k-\ell}.
\end{equation}
\end{proposition}
\begin{proof}
The result is, e.g., a special case of~\cite[Problem~3.6]{Higham2008}.
\end{proof}

With help of Proposition~\ref{prop:frechet_derivative_polynomial}, we can prove the following result about the nonzero structure of $L_{p_m}(A, E_{ij})$. 

\begin{lemma}\label{lem:sparsity_pattern_polynomial}
Let $p_m \in \Pi_m$, let $A \in \Rnn$ and let $E_{ij} = \ve_i\ve_j^T \in \Rnn$. Then
\begin{equation*}
    [L_{p_m}(A, E_{ij})]_{uv} = 0 \quad\text{ if } d(u,i) + d(j,v) \geq m.
\end{equation*}
\end{lemma}
\begin{proof}
We begin by recalling that $[A^\ell]_{rs} = 0$ if $d(r,s) > \ell$. Now consider formula~\eqref{eq:frechet_derivative_polynomial} for the special case $E = E_{ij}$, giving
\begin{eqnarray*}
   [L_{p_m}(A, E_{ij})]_{uv} &=& \ve_u^T\left(\sum\limits_{k=1}^m \alpha_k \sum\limits_{\ell=1}^k A^{\ell-1}E_{ij}A^{k-\ell}\right)\ve_v \nonumber\\
   &=& \sum\limits_{k=1}^m \alpha_k \sum\limits_{\ell=1}^k \ve_u^TA^{\ell-1}\ve_i\ve_j^TA^{k-\ell}\ve_v\nonumber\\
   &=& \sum\limits_{k=1}^m \alpha_k \sum\limits_{\ell=1}^k [A^{\ell-1}]_{ui}[A^{k-\ell}]_{jv}.
\end{eqnarray*}
It suffices to consider the term for $k = m$. A term in the inner sum can only be nonzero if both $[A^{\ell-1}]_{ui}$ and $[A^{m-\ell}]_{jv}$ are nonzero, i.e., if $d(u,i) \leq \ell-1$ and $d(j,v) \leq m-\ell$. If $d(u,i) + d(j,v) \geq m$, these two inequalities cannot both be satisfied at the same time for any $\ell$, so that all terms appearing in the sum are zero. The assertion of the lemma directly follows.
\end{proof}

Similarly, the special structure of the direction term $E_v$ from~\eqref{eq:modification_node_removal} allows to conclude about the sparsity pattern of the Fr\'echet derivative $L_{p_m}(A,E_v)$.

\begin{lemma}\label{lem:sparsity_pattern_polynomial_nodes}
Let $p_m \in \Pi_m$, let $A \in \Rnn$ and let $E_v = -(\ve_v \va_{v:} + \va_{:v}\ve_v^T) \in \Rnn$. Then
\begin{equation*}
    [L_{p_m}(A, E_v)]_{u_1u_2} = 0 \quad\text{ if } d(u_1,v) + d(v,u_2) \geq m+1.
\end{equation*}
\end{lemma}
\begin{proof}
As the Fr\'echet derivative is linear in its second argument, we have 
\begin{equation}\label{eq:frechet_linearity}
[L_{p_m}(A, E_v)]_{u_1u_2} = - L_{p_m}(A, \ve_v \va_{v:}) - L_{p_m}(A, \va_{:v}\ve_v^T).
\end{equation}
In particular, $L_{p_m}(A, E_v)]_{u_1u_2}$ is zero when both individual terms on the right-hand side of~\eqref{eq:frechet_linearity} are zero. Using Proposition~\ref{prop:frechet_derivative_polynomial} and proceeding analogously to the proof of Lemma~\ref{lem:sparsity_pattern_polynomial}, we find
\begin{equation}\label{eq:frechet_derivative_polynomial_Ei1}
 [L_{p_m}(A, \ve_v \va_{v:})]_{u_1u_2} = \sum\limits_{k=1}^m \alpha_k \sum\limits_{\ell=1}^k [A^{\ell-1}]_{u_1v}[A^{k-\ell+1}]_{vu_2}
\end{equation}
and 
\begin{equation}\label{eq:frechet_derivative_polynomial_Ei2}
 [L_{p_m}(A, \va_{:v}\ve_v^T)]_{u_1u_2} = \sum\limits_{k=1}^m \alpha_k \sum\limits_{\ell=1}^k [A^{\ell}]_{u_1v}[A^{k-\ell}]_{vu_2}. 
\end{equation}
Clearly, when $d(u_1,v) + d(v,u_2) \geq m+1$, then all terms appearing in the sums in~\eqref{eq:frechet_derivative_polynomial_Ei1} and~\eqref{eq:frechet_derivative_polynomial_Ei2} are zero, from which the assertion of the lemma follows.
\end{proof}

\subsection{A priori bounds for network sensitivity}
Using the results from Section~\ref{subsec:sparsity_frechet}, together with a recent result from~\cite{CrouzeixKressner2020}, we can obtain bounds for the entries of $L_f(A_G, E_{ij})$ and $L_f(A_G, E_v)$ from best polynomial approximation of $f^\prime$ on the numerical range (or field of values) of $A_G$,
\begin{equation*}
    W(A_G) = \{ \vx^T\!\!A_G \vx : \|\vx\| = 1\}.
\end{equation*}

\begin{theorem}\label{the:decay_derivative}
Let $A_G \in \Rnn, E_{ij} = \ve_i\ve_j^T \in \Rnn$ and denote by $W(A_G)$ the numerical range of $A_G$. Then
\begin{equation*}
| [L_f(A_G, E_{ij})]_{uv} | \leq C\cdot \min_{p \in \Pi_{m(u,v)-1}}\ \max_{z \in W(A_G)} |f^\prime(z) - p(z)|
\end{equation*}
where $m(u,v) = d(u,i) + d(j,v)$ and
\begin{equation}\label{eq:C}
    C = \begin{cases}
    1 & \text{if $A$ is normal,} \\
    \left(1+\sqrt{2}\right)^2 & \text{otherwise.} \\
    \end{cases}
\end{equation}
\end{theorem}

\begin{remark}\label{rem:decay_node_removal}
If one replaces $E_{ij}$ with $E_v$, by following similar steps as in the proof of Theorem~\ref{the:decay_derivative}, one finds the bound
\begin{equation*}
|[L_f(A_G, E_v)]_{u_1u_2} | \leq C \cdot \sqrt{\deg(v)} \min_{p \in \Pi_{m(u_1,u_2)-1}}\ \max_{z \in W(A_G)} |f^\prime(z) - p(z)|
\end{equation*}
where $\deg(v) := \sum_{u=1}^n w_{vu}$ denotes the ``weighted degree'' of node $v$ (also known as the ``strength'' of $v$) and $m(u_1,u_2) := d(u_1,v) + d(v,u_2)+1$.
\end{remark}

In the special case of the exponential function $f(z) = e^z$ that we are most interested in, we have $f^\prime(z) = f(z)$, so that any polynomial approximation result for $f$ can directly be used to obtain decay bounds for the Fr\'echet derivative. We demonstrate one specific bound obtained this way in the following corollary, which results from combining Theorem~\ref{the:decay_derivative} with~\cite[Lemma 2]{KandolfKoskelaReltonSchweitzer2021}. Note that of course any other polynomial approximation result for the exponential function, like, e.g.,~\cite[Corollary~4.1, Corollary~4.2]{BeckermannReichel2009} could also be used in conjunction with Theorem~\ref{the:decay_derivative} to obtain explicit bounds for the Fr\'echet derivative. The proof of this result is presented in Appendix~\ref{sec:appendix_proofs}.

\begin{corollary}\label{cor:decay_undirected}
Let $G$ be an undirected graph with adjacency matrix $A_G$ and denote the smallest and largest eigenvalue of $A_G$ by $\lmin$ and $\lmax$, respectively. Further, let $E_{ij} = \ve_i\ve_j^T \in \Rnn$ and denote $m(u,v) := d(u,i) + d(j,v)$.

Then, if $\sqrt{\lmax-\lmin} \leq m(u,v) \leq \dfrac{\lmax-\lmin}{2}$, we have the bound
\begin{equation}\label{eq:decay_bound_undirected1}
|[L_{\exp}(A_G, E_{ij})]_{uv}| \leq 2 \frac{\lmax-\lmin}{m(u,v)} e^{\lmax-\frac{4m(u,v)^2}{5(\lmax-\lmin)}}
\end{equation}
and if $m(u,v) > \dfrac{\lmax-\lmin}{2}$, we have the bound
\begin{equation}\label{eq:decay_bound_undirected2}
|[L_{\exp}(A_G, E_{ij})]_{uv}| \leq 8\frac{e^{\lmax}\cdot m(u,v)}{\lmax-\lmin}  \left(\frac{e\cdot(\lmax-\lmin)}{4m(u,v)+2(\lmax-\lmin)}  \right)^{m(u,v)}\!\!\!\!\!\!\!\!\!\!\!\!\!\!.
\end{equation}
\end{corollary}

\begin{remark}
The result of Corollary~\ref{cor:decay_undirected} can be used to obtain a priori estimates for the sensitivity of individual nodes with respect to modifications in an edge $(i,j)$. In particular, for the sensitivity of subgraph centrality, an estimate is directly obtained by setting $v = u$ in~\eqref{eq:decay_bound_undirected1} or~\eqref{eq:decay_bound_undirected2}. This mathematically confirms the intuition that nodes are more sensitive to the modification of ``nearby'' edges than to modifications of edges in other parts of the network. It is very similar in spirit to the analysis performed in~\cite{PozzaTudisco2018}, where analogous bounds were obtained directly for the change of centrality scores instead of for their sensitivity.
\end{remark}

\begin{remark}
If the extremal eigenvalues $\lmin, \lmax$ of $A_G$ are not known, one can still obtain decay estimates by inserting suitable bounds. As a simple example, by Ger\v{s}gorin's disk theorem we have for any adjacency matrix of an undirected graph that $\spec(A_G) \subset [-\deg_{\max}, \deg_{\max}]$, where $\deg_{\max}$ denotes the maximum degree of any node in $G$. Thus $\lmax-\lmin \leq 2 \deg_{\max}$, and for graphs where all nodes have similar degrees (which is, e.g., the case for grid-like graphs and many road networks), the decay estimates obtained from using this bound might still carry meaningful information. In graphs with highly varying degrees, the estimates obtained this way will typically not accurately capture the actual decay behavior.
\end{remark}

It is also possible to obtain decay bounds for $L_{\exp}(A_G, E_{ij})$ in the nonsymmetric case, i.e., for directed graphs. In this case, $W(A_G)$ is not an interval but an arbitrary convex set in the complex plane. Depending on the shape of this set (or the shape of a larger set containing it), many different bounds can be obtained, typically in terms of conformal mappings and Faber polynomials; see, e.g.,~\cite{PozzaTudisco2018,BeckermannReichel2009,HochbruckLubich1997} for examples of this technique (mostly in the context of analyzing convergence of Krylov subspace methods instead of finding decay bounds).

In the following corollary, we demonstrate the bounds arising from the assumption that $W(A_G)$ is contained in a disk of radius $r$. One can always take a disk centered at the origin and choose $r = \nu(A_G)$, the numerical radius of $A_G$, i.e., the largest eigenvalue of the Hermitian matrix $\frac12 (A_G+A_G^T)$. When the extent of $W(A_G)$ is not (close to) symmetric to the imaginary axis, better bounds might be achieved by choosing the center of the disk to be different from the origin. The result is stated in terms of the lower incomplete gamma function
\[
\gamma(a,x) = \int_0^x t^{a-1}e^{-t} \d t.
\]
Its proof is given in Appendix~\ref{sec:appendix_proofs}.
\begin{corollary}\label{cor:decay_directed}
Let $G$ be a digraph with adjacency matrix $A_G$ and assume that $W(A_G)$ is contained in a disk of radius $r$ centered at $c$. Further, let $E_{ij} = \ve_i\ve_j^T \in \Rnn$ and denote $m(u,v) := d(u,i) + d(j,v)$. Then
\begin{equation}\label{eq:decay_bound_directed}
|L_{\exp}(A, E_{ij})_{uv}| \leq 2\left(1+\sqrt{2}\right)^2 e^{r+c}\frac{\gamma(m(u,v),r)}{(m(u,v)-1)!}.
\end{equation}
\end{corollary}

\begin{remark}\label{rem:simpler_bound}
For $a \geq 1$, the lower incomplete gamma function fulfills
\[
\gamma(a,x) \leq \left(1-e^{-x}\right)\frac{x^{a-1}}{a};
\]
see, e.g.,~\cite[eq.~8.10.2]{Nist2010}. Thus, when $m(u,v) \geq 1$, we can replace~\eqref{eq:decay_bound_directed} by the easier to grasp bound
\begin{align}
|L_{\exp}(A, E_{ij})_{uv}| \leq 2\left(1+\sqrt{2}\right)^2 \left(e^{r+c}-1\right) \frac{r^{m(u,v)-1}}{m(u,v)!}. \label{eq:simpler_bound_directed}
\end{align}
Note, however, that the bound~\eqref{eq:simpler_bound_directed} is actually \emph{increasing} in $m(u,v)$ as long as $m(u,v) \leq r$, which is not the case for~\eqref{eq:decay_bound_directed}.
\end{remark}

Next, we state two results with a priori bounds for sensitivity with respect to node removal, which are essentially analogues of Corollary~\ref{cor:decay_undirected} and~\ref{cor:decay_directed} based on the modification given in Remark~\ref{rem:decay_node_removal}. As the lines of argument are analogous to before we just state the final results and refrain from providing all details.

\begin{corollary}\label{cor:decay_undirected_nodes}
Let $G$ be an undirected graph with adjacency matrix $A_G$ and denote the smallest and largest eigenvalue of $A_G$ by $\lmin$ and $\lmax$, respectively. Further, let $E_{v} = -(\ve_v \va_{v:} + \va_{:v}\ve_v^T) \in \Rnn$ and denote $m(u_1,u_2) := d(u_1,v) + d(v,u_2)+1$. Then, if $\sqrt{\lmax-\lmin} + 1 \leq m(u_1,u_2) \leq \dfrac{\lmax-\lmin}{2} + 1$, we have the bound
\begin{equation}\label{eq:decay_bound_undirected_nodes1}
|[L_{\exp}(A_G, E_v)]_{u_1u_2}| \leq 2 \sqrt{\deg(v)} \frac{\lmax-\lmin}{m(u_1,u_2)-1} e^{\lmax-\frac{(m(u_1,u_2)-1)^2}{\frac{5}{4}(\lmax-\lmin)}}
\end{equation}
and if $m(u_1,u_2) > \dfrac{\lmax-\lmin}{2}+1$, we have the bound
\begin{align}
&\phantom{\leq} |[L_{\exp}(A_G, E_v)]_{u_1u_2}| \nonumber\\
&\leq 8 \sqrt{\deg(v)}\frac{e^{\lmax}(m(u_1,u_2)-1)}{\lmax-\lmin} \cdot\left(\frac{e\cdot(\lmax-\lmin)}{4(m(u_1,u_2)-1)+2(\lmax-\lmin)}  \right)^{m(u_1,u_2)-1}\!\!\!\!\!\!\!\!\!\!\!\!\!\!\!\!\!\!\!\!\!\!\!\!.\label{eq:decay_bound_undirected_nodes2}
\end{align}
\end{corollary}

\begin{corollary}\label{cor:decay_directed_nodes}
Let $G$ be a digraph with adjacency matrix $A_G$ and assume that $W(A_G)$ is contained in a disk of radius $r$ centered at $c$. Further, let $E_{v} = -(\ve_v \va_{v:} + \va_{:v}\ve_v^T) \in \Rnn$ and denote $m(u_1,u_2) := d(u_1,v) + d(v,u_2)+1$. Then
\begin{equation*}
|L_{\exp}(A_G, E_{v})_{u_1u_2}| \leq 2\left(1+\sqrt{2}\right)^2\sqrt{\deg(v)} e^{r+c}\frac{\gamma(m(u,v),r)}{(m(u_1,u_2)-1)!}.
\end{equation*}
\end{corollary}

To conclude this section, we compare the quality of our decay bounds to the actual sensitivity values for a real-world network.

\begin{example}\label{ex:decay_london_transport}
\begin{figure}
    \centering
    \includegraphics[width=.7\textwidth]{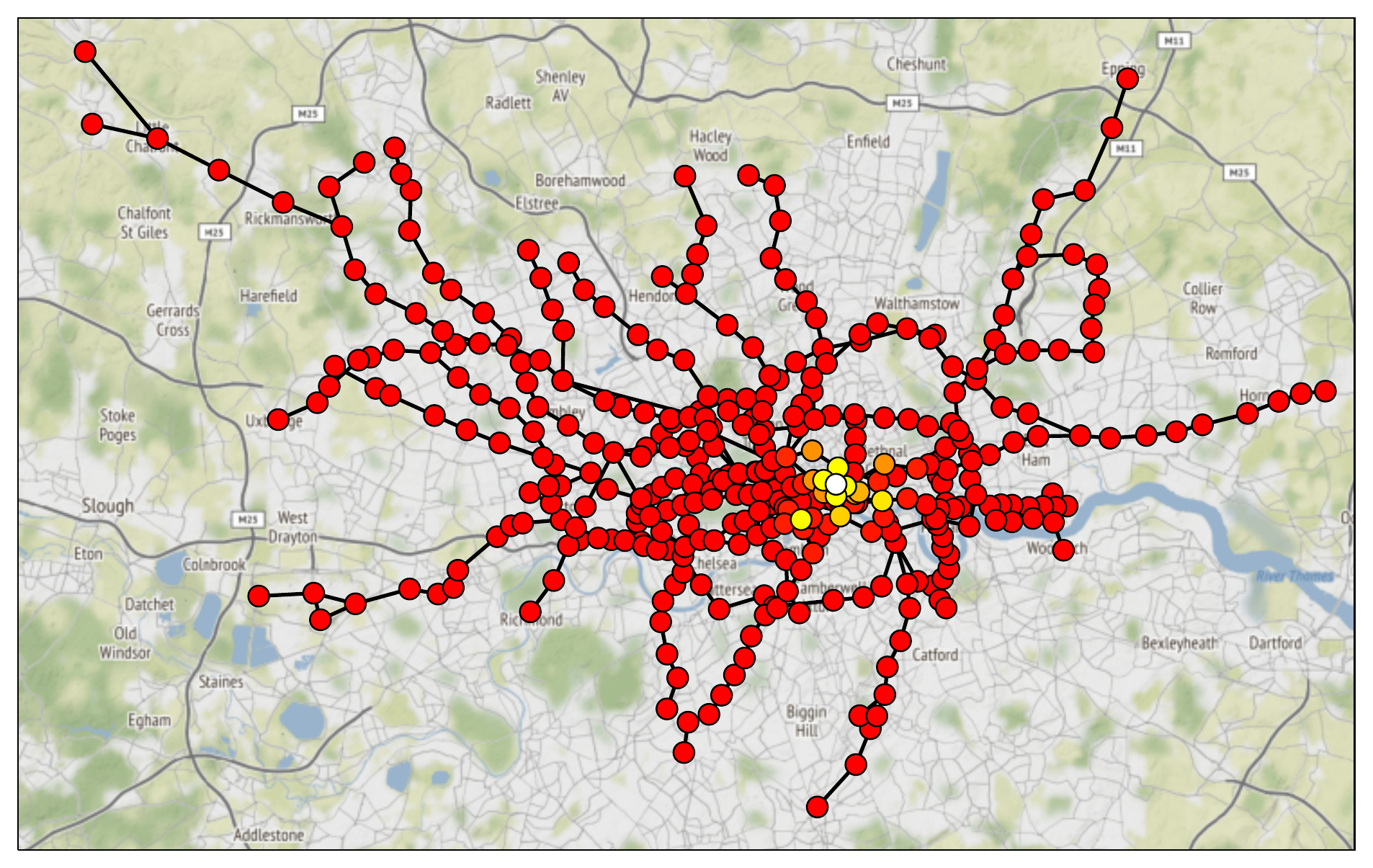}
    \caption{Bounds for the sensitivity of nodes in the London transportation network with respect to the removal of the node corresponding to \emph{Moorgate station} (depicted in white), obtained from~\eqref{eq:decay_bound_undirected_nodes1} and~\eqref{eq:decay_bound_undirected_nodes2}. Light colors correspond to high sensitivity, while dark colors correspond to low sensitivity (street map generated with \texttt{cartopy}~\cite{Cartopy}, map data \textcopyright\ OpenStreetMap).}
    \label{fig:decay_london_graph}
\end{figure}

This example illustrates the result of Corollary~\ref{cor:decay_undirected_nodes}, using a graph representing the London city transportation network~\cite{DeDomenicoDataset,DeDomenicoSoleRibaltaGomezArenas2014}. Each node of the network corresponds to a station and edges between stations indicate train, metro or bus connections. The original dataset~\cite{DeDomenicoDataset} is actually a multilayer network, from which we obtain an undirected graph by aggregating the different layers. The maximum degree of any node in the resulting network is $7$, and its adjacency matrix has spectrum $\spec(A_G) = [-3.24, 3.79]$ (rounded to two decimal digits).

We consider the network modification resulting from the removal of the node corresponding to \emph{Moorgate station} near the center of London and bound the sensitivity of the nodes' subgraph centrality by using~\eqref{eq:decay_bound_undirected_nodes1} and~\eqref{eq:decay_bound_undirected_nodes2} for $u_1 = u_2$. The resulting sensitivity bounds are illustrated by the color-coding in Figure~\ref{fig:decay_london_graph} (the node depicted in white is the removed node, Moorgate station). Only nodes in the direct surrounding of Moorgate station are sensitive to this modification, while the influence rapidly drops off with increasing distance (as one would also intuitively expect). Note that for nodes $u$ with a distance of one to Moorgate station, $m(u,u)$ fulfills neither of the inequalities in Corollary~\ref{cor:decay_undirected_nodes}, so that no bound can be obtained. Whenever this happens, one can of course expect the corresponding nodes to be highly sensitive to the modification at hand (and therefore, we also depict those nodes in light colors in Figure~\ref{fig:decay_london_graph}). Thus, if one is only interested in finding all sensitive nodes, then this is no limitation, but one does not obtain an actual bound for quantifying the influence.

\begin{figure}
    \centering
    \pgfplotsset{height=0.38\linewidth,width=0.8\linewidth,compat=1.10,every axis/.append style={legend style={/tikz/every even column/.append style={column sep=6pt}}}}
    \pgfplotsset{every tick label/.append style={font=\scriptsize}}

\noindent%
    \input{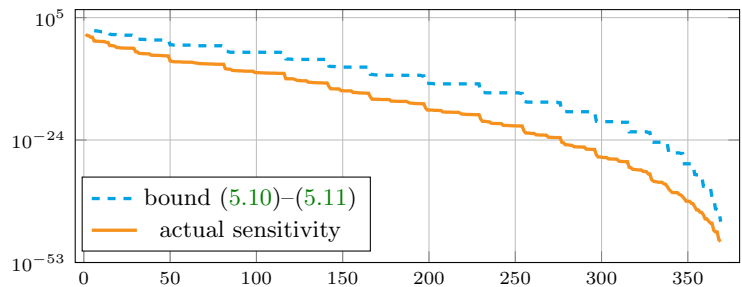}
   
    \caption{Comparison of the bounds~\eqref{eq:decay_bound_undirected_nodes1} and~\eqref{eq:decay_bound_undirected_nodes2} to the actual sensitivity of the nodes subgraph centrality. Nodes are reordered according to their sensitivity (descendingly).}
    \label{fig:decay_london_comparison}
\end{figure}

To gauge how accurately the bounds~\eqref{eq:decay_bound_undirected_nodes1} and~\eqref{eq:decay_bound_undirected_nodes2} capture the actual sensitivity of the nodes' subgraph centralities, we explicitly compute the sensitivities of subgraph centrality of all nodes and compare the obtained values to our bounds; see Figure~\ref{fig:decay_london_comparison}, where we have sorted the nodes decreasingly by their actual sensitivity for better visualization. We can observe that our bounds capture the qualitative behavior of the sensitivity very well, but that the magnitude of the sensitivities is overestimated by some margin (a phenomenon that is also well-known for similar decay bounds for entries of matrix functions). Clearly, the bounds~\eqref{eq:decay_bound_undirected_nodes1} and~\eqref{eq:decay_bound_undirected_nodes2} attain the same value for all nodes having the same distance from Moorgate station, leading to a ``staircase-like'' shape of the bound in Figure~\ref{fig:decay_london_comparison}. Interestingly, we observe a similar pattern, although a little less pronounced, in the actual sensitivities, showing that this is not simply an artifact of the technique used for finding the decay bounds, but a feature that is actually observable in the exact values. 
\end{example}

\section{Concluding remarks}\label{sec:conclusions}
We have proposed a computational procedure for identifying network modifications to which the communicability of the network is most sensitive. For typical real-world networks, the computational complexity of the method scales linearly with the number of nodes, making it feasible also for large scale networks.

We have also extended the concept of network sensitivity with respect to edge modifications from total communicability to subgraph centrality and the Estrada index and we have demonstrated how sensitivity with respect to removal of nodes fits into the framework. 

Additionally, we have derived a priori bounds for the sensitivities (based on sparsity patterns of the Fr\'echet derivative of polynomial matrix functions with structured direction terms), which predict the actual qualitative behavior of sensitivity quite well and give some further intuitive insight into the concept of network sensitivity. These decay bounds might also be of independent interest in other applications where Fr\'echet derivatives with structured direction terms occur. 

It is an interesting topic for future research to compare the edge and node rankings obtained with the sensitivity concept to those obtained by other means for actual analysis of real-world networks (e.g., in the context of vulnerability analysis) and identify the practical advantages and disadvantages of each approach. Another research avenue, which is not related to the analysis of complex networks, is identifying further application areas in which Fr\'echet derivatives with respect to structured, low-rank direction terms play an important role, which could, e.g., benefit from the decay bounds developed in Section~\ref{sec:decay_bounds}.

\paragraph{Acknowledgments} The author wishes to thank the anonymous referees for their helpful comments which helped improve the manuscript.

\appendix
\section{Technical proofs}\label{sec:appendix_proofs}
\begin{proof}[Proof of Theorem~\ref{the:decay_derivative}]
Let $u$ and $v$ be fixed arbitrarily and let $p$ be any polynomial of degree at most $m(u,v)$. Then, by Lemma~\ref{lem:sparsity_pattern_polynomial}, we have $[L_p(A_G, E_{ij})]_{uv} = 0$. Consequently, 
\begin{equation}\label{eq:proof_theorem1}
[L_f(A_G, E_{ij})]_{uv} = [L_f(A_G, E_{ij})]_{uv} - [L_p(A_G, E_{ij})]_{uv} = [L_g(A_G, E_{ij})]_{uv},
\end{equation}
where $g = f-p$. We now bound the absolute value of the right-hand side of~\eqref{eq:proof_theorem1} as
\begin{equation}\label{eq:proof_theorem2}
| [L_g(A_G, E_{ij})]_{uv} | \leq \|L_g(A_G, E_{ij})\| \leq \|L_g(A_G, \cdot)\| \cdot \|E_{ij}\|.
\end{equation}
By \cite[Section~1 and Corollary~5.1]{CrouzeixKressner2020}, we have
\begin{equation*}
\|L_g(A_G, \cdot)\| \leq C \cdot \max_{z \in W(A_G)} |g^\prime(z)|,
\end{equation*}
where $C$ is defined in~\eqref{eq:C}. As $g^\prime = f^\prime - p^\prime$ and $p^\prime \in \Pi_{m(u,v)-1}$, we obtain
\begin{equation}\label{eq:proof_theorem3}
\|L_g(A_G, \cdot)\| \leq C \cdot \min_{p \in \Pi_{m(u,v)-1}}\ \max_{z \in W(A_G)} |f^\prime(z) - p(z)|,
\end{equation}
as $p$ can be chosen arbitrarily. Combining~\eqref{eq:proof_theorem1},~\eqref{eq:proof_theorem2} and~\eqref{eq:proof_theorem3} and using $\|E_{ij}\| = 1$ proves the result.
\end{proof}

\begin{proof}[Proof of Corollary~\ref{cor:decay_undirected}]
It is well known that $\exp(A_G + \sigma I) = \exp(\sigma)  \cdot \exp(A_G)$ for any $\sigma$, and this relation readily carries over to the Fr\'echet derivative, so that we also have 
\begin{equation}\label{eq:multiplicative_factor_frechet}
L_{\exp}(A_G + \sigma I, E_{ij}) = e^{\sigma} \cdot L_{\exp}(A_G, E_{ij}).
\end{equation} 
Define the negative semidefinite matrix $\widetilde{A}_G := A_G - \lmax I$ with spectrum in $[-\lmax+\lmin, 0]$. By~\eqref{eq:multiplicative_factor_frechet}, we then have 
\begin{equation}\label{eq:multiplicative_constant_Atilde}
[L_{\exp}(A_G, E_{ij})]_{uv} = e^{\lmax} \cdot [L_{\exp}(\widetilde{A}_G, E_{ij})]_{uv}.
\end{equation}
The entries of $L_{\exp}(\widetilde{A}_G, E_{ij})$ can now be bounded via Theorem~\ref{the:decay_derivative}, using the polynomial approximation result of~\cite[Lemma~2]{KandolfKoskelaReltonSchweitzer2021} for positive semidefinite matrices. This result states that
\begin{equation*}
\min_{p \in \Pi_{m-1}}\ \max_{z \in [-4\rho, 0]} |\exp(z)-p(z)| \leq \begin{cases}
8 \frac{\rho}{m} e^{-\frac{m^2}{5\rho}}, & \text{ if } \sqrt{4 \rho } \leq m \leq 2 \rho ,\\[2ex]
2 \frac{m}{\rho}  \left(\frac{e\rho}{m+2\rho}  \right)^m, &  \text{ if } m > 2 \rho.
\end{cases}
\end{equation*}
Inserting $\rho = (\lmax-\lmin)/4$ and combining with~\eqref{eq:multiplicative_constant_Atilde} proves the result.
\end{proof}

\begin{proof}[Proof of Corollary~\ref{cor:decay_directed}]
Let $\Delta := \Delta_{r,c}$ denote the closed disk of radius $r > 1$ centered at $c$. Denote by $\phi$ the conformal mapping from the exterior of $\Delta$ onto the exterior of the unit disk and by $\psi$ its inverse. Clearly, $\phi(z) = (z-c)/r$ and $\psi(w) = rw+c$. As $W(A_G) \subseteq \Delta$, it follows from Theorem~\ref{the:decay_derivative} that
\begin{equation}\label{eq:proof_decay_directed_new1}
|L_{\exp}(A_G, E_{ij})_{uv}| \leq \left(1+\sqrt{2}\right)^2 \cdot \min_{p \in\Pi_{m(u,v)-1}}\ \max_{z \in \Delta} |e^{z} - p(z)|.
\end{equation}
It is well-known (see, e.g.,~\cite[Section~2,~3]{BeckermannReichel2009}) that the right-hand side of~\eqref{eq:proof_decay_directed_new1} can be bounded in terms of the Faber coefficients of the exponential function, which gives
\begin{equation}\label{eq:proof_decay_directed_new2}
|L_{\exp}(A_G, E_{ij})_{uv}| \leq 2\left(1+\sqrt{2}\right)^2 \sum\limits_{k=m(u,v)}^{\infty}|f_k|
\end{equation}
where
\[
f_k = \frac{1}{2\pi i} \int_{|w| = R} \frac{e^{\psi(w)}}{w^{k+1}}  \d w.
\]
for any $R > 1$. Due to the simple nature of $\psi$, we can explicitly compute the Faber coefficients by the residue theorem, which gives
\begin{equation}\label{eq:proof_decay_directed_new3}
f_k = \frac{1}{2\pi i} \int_{|w| = R} \frac{e^{\psi(w)}}{w^{k+1}}  \d w = \frac{r^ke^c}{k!}.
\end{equation}
Further, we have the following relation for the lower incomplete gamma function, 
\begin{equation}\label{eq:bound_gamma_sum}
\sum\limits_{k = a}^\infty \frac{x^k}{k!} = e^x\frac{\gamma(a,x)}{(a-1)!}.
\end{equation}
which can, e.g., be obtained from~\cite[eq.~(1.7)--(1.8)]{JonesThron1985} by simple algebraic manipulations. Inserting~\eqref{eq:proof_decay_directed_new3} and~\eqref{eq:bound_gamma_sum} into~\eqref{eq:proof_decay_directed_new2} then yields
\[
|L_{\exp}(A_G, E_{ij})_{uv}| \leq 2\left(1+\sqrt{2}\right)^2 \sum\limits_{k=m(u,v)}^{\infty}\frac{r^ke^c}{k!} = 2\left(1+\sqrt{2}\right)^2 e^{r+c}\frac{\gamma(m(u,v),r)}{(m(u,v)-1)!},
\]
which completes the proof.
\end{proof}

\bibliography{matrixfunctions}
\bibliographystyle{abbrv}

\end{document}